\documentclass[a4paper,final]{amsart}

\usepackage{amsmath,amsthm,amsfonts,amssymb,mathrsfs}
\usepackage[notcite,notref]{showkeys}

\numberwithin{equation}{section}

\newtheorem{theorem}{Theorem}[section]
\newtheorem{lemma}[theorem]{Lemma}

\newtheorem{prop}[theorem]{Proposition}
\newtheorem{defin}[theorem]{Definition}

\theoremstyle{remark}
\newtheorem{rem}[theorem]{Remark}

\def\R{\mathbb{R}}
\def\U{\mathcal{U}}
\def\Da{\Lambda^\alpha}
\newcommand{\rf}[1]{{\rm (}\ref{#1}{\rm )}}
\def\Proof{\begin{proof}}
\def\endProof{\end{proof}}
\def\ue{u^\varepsilon}

\def\eps{\varepsilon}

\begin{document}

\title[Fractal Burgers equation]
{Asymptotic properties of entropy solutions to fractal Burgers equation}

\author{Nathael Alibaud}

\address{N. Alibaud: Laboratoire de Math\'ematiques de Besan\c{c}on,
  UMR CNRS 6623, Universit\'e de Franche-Comt\'e, UFR Sciences et techniques, 
16 route de Gray, 25030
  Besan\c con cedex, France} \email{Nathael.Alibaud@ens2m.fr}
\urladdr{http://www-math.univ-fcomte.fr/pp\_Annu/NALIBAUD}

\author{Cyril Imbert}

\address{
C. Imbert: CEREMADE, UMR CNRS 7534, Universit\'e Paris-Dauphine, 
Place de Lattre de Tassigny, 75775 Paris cedex
16, France}

\email{imbert@ceremade.dauphine.fr}
\urladdr{http://www.ceremade.dauphine.fr/~imbert}

\author{Grzegorz Karch}

\address{G. Karch: Instytut Matematyczny,
Uniwersytet Wroc{\l}awski, 
 pl.~Grunwaldzki 2/4, 50--384 Wroc{\l}aw, Poland}
\email{\tt Grzegorz.Karch@math.uni.wroc.pl}
\urladdr{http://www.math.uni.wroc.pl/~karch}

\begin{abstract} 
  We study properties of solutions of the initial value problem for
  the nonlinear and nonlocal equation $u_t+(-\partial^2_x)^{\alpha/2}
  u+uu_x=0$ with $\alpha\in (0,1]$, supplemented with an initial datum
  approaching the constant states $u_\pm$ ($u_-<u_+$) as $x\to
  \pm\infty$, respectively.  It was shown by Karch, Miao \& Xu (SIAM
  J. Math. Anal. {\bf 39} (2008), 1536--1549) that, for $\alpha\in
  (1,2)$, the large time asymptotics of solutions is described by
  rarefaction waves.  The goal of this paper is to show that the
  asymptotic profile of solutions changes for $\alpha\leq 1$.  If
  $\alpha=1$, there exists a self-similar solution to the equation
  which describes the large time asymptotics of other solutions.  In
  the case $\alpha \in (0,1)$, we show that the nonlinearity of the
  equation is negligible in the large time asymptotic expansion of
  solutions.
\end{abstract}



\date{\today}

\thanks{The authors would like to thank the referee for suggestions
  that improved significantly the presentation of the results.  The
  first author would like to thank the Department of Mathematics
  Prince of Songkla University (Hat Yai campus, Thailand) for having
  ensured a large part of his working facilities.  The second author
  was partially supported by the ANR project ``EVOL''.  The work of
  the third author was partially supported by the European Commission
  Marie Curie Host Fellowship for the Transfer of Knowledge ``Harmonic
  Analysis, Nonlinear Analysis and Probability'' MTKD-CT-2004-013389,
  and by the Polish Ministry of Science grant N201 022 32/0902.  The
  authors were also partially supported by PHC-Polonium project no
  20078TL "Nonlinear evolutions equations with anomalous diffusions''.
}

\maketitle

\section{Introduction}

In this work, we continue the study of asymptotic properties of solutions
of the Cauchy problem for the following nonlocal conservation law
\begin{eqnarray}\label{e1.1}
&&u_t+\Da u +uu_x=0, \quad x\in\R,\; t>0,\\
\label{e1.2} &&u(0,x)=u_0(x),
\end{eqnarray}
where
$\Da=(-\partial^2/\partial x^2)^{\alpha/2}$
is the pseudodifferential operator defined via the Fourier
transform $\widehat {(\Da v)}(\xi)= |\xi|^\alpha \, \widehat v(\xi).$
This equation is referred to as the fractal Burgers equation.

The initial datum~$u_0 \in L^\infty(\R)$ is assumed to satisfy:
\begin{equation}\label{ini as}
\exists u_-<u_+  \mbox{ with }
u_0-u_-\in L^1(-\infty,0) \mbox{ and }
 u_0-u_+\in L^1(0,+\infty)
\end{equation}
(where~$u_\pm$ are real numbers).
An interesting situation is where~$u_0 \in BV(\R)$, that is to say 
\begin{equation}\label{u0}
u_0(x)=c+\int_{-\infty}^x m(dy)
\end{equation}
with $c \in \R$ and a finite signed measure~$m$ on $\R$. In that case 
Jourdain, M\'el\'eard, and Woyczy\'nski
\cite{JMW,JMW1} have recently given a probabilistic interpretation to 
problem~\rf{e1.1}--\rf{e1.2}. Assumption~\eqref{ini as} holds true when
\begin{equation}\label{u:pm}
u_-=c \quad \mbox{and} \quad  u_+-u_-=\int_{\R} m(dx)>0.
\end{equation}
If $c=0$ and if $m$ is a probability measure, the function
$u_0$ defined in \rf{u0} is the {\it cumulative distribution
  function} and this property is shared by the solution $u(t) \equiv u(\cdot,t)$ for
every $t>0$ (see~\cite{JMW,JMW1}).  As a consequence of our results,
we describe the asymptotic behavior of the family $\{u(t)\}_{t\geq 0}$
of probability distribution functions as $t\to+\infty$ (see the summary
at the end of this section).

It was shown in \cite{kmx} that, under assumptions \rf{ini as}--\rf{u:pm} and for $1<\alpha \leq 2$, the large time asymptotics of
solution to \rf{e1.1}--\rf{e1.2} is described by the so-called {\it
  rarefaction waves}. The goal of this paper is to complete these
results and to obtain  universal asymptotic profiles of solutions for
$0<\alpha\leq 1$.

\subsection{Known results}
Let us first  recall the results obtained in \cite{kmx}.
For~$\alpha \in (1,2]$, the initial value problem for the fractal
Burgers equation \rf{e1.1}--\rf{e1.2} with $u_0\in L^\infty(\R)$ has
a unique, smooth, global-in-time solution  ({\it cf.}
{\cite[Thm.~1.1]{dgv}}, {\cite[Thm.~7]{dri}}).  If,
moreover, the initial datum is of the form \rf{u0} and satisfies
\rf{ini as}--\rf{u:pm}, the corresponding solution $u$
behaves asymptotically when $t\to+\infty$ as the rarefaction wave ({\it cf.}
\cite[Thm.~1.1]{kmx}). More precisely, for every $p\in
(\frac{3-\alpha}{\alpha-1}, +\infty]$ there exists a constant $C>0$ such that for
all $t>0$,
\begin{equation}\label{asymp:raref}
\|u(t)-w^R(t)\|_{p}\leq
Ct^{-\frac12[\alpha-1-\frac{3-\alpha}{p}]}\log (2+t)
\end{equation}
($\|\cdot\|_p$ is the standard norm in $L^p(\R)$).
Here, the rarefaction wave 
 is the explicit (self-similar) function
\begin{equation}\label{rarefaction}
w^R(x,t)=w^R \left(\frac{x}{t},1\right) \equiv \left\{
\begin{aligned}
&u_-\,,\quad &&\frac{x}t\leq u_-,\\
&\frac{x}t\,,\quad &&u_-\leq \frac{x}t\leq u_+, \\
& u_+\,,\quad &&\frac{x}t\geq u_+.
\end{aligned}
\right.
\end{equation}
It is well-known that $w^R$ is the unique entropy solution of the
 Riemann problem for the nonviscous Burgers equation
$w^R_t+w^R w^R_x=0$.

The goal of the work is to show that, for $\alpha\in (0,1]$, one
should expect completely different asymptotic profiles of
solutions. Let us notice that the initial value problem
\rf{e1.1}--\rf{e1.2} has a unique global-in-time entropy solution
for every $u_0\in L^\infty(\R)$ and $\alpha\in (0,1]$ due to the
recent work by the first author \cite{A07}. We recall this result in Section
\ref{sec:entropy}.

\subsection{Main results.}
Our two main results are Theorems~\ref{thm:a<1} and
\ref{thm:a=1}, stated below.  Both of them are a consequence of the
following $L^p$-estimate of the difference of two entropy solutions.
\begin{theorem}
\label{thm:as-stab-fb} Let $0< \alpha \leq 1$.
  Assume that $u$ and $\widetilde u$ are two entropy solutions of
  \rf{e1.1}--\rf{e1.2} with initial conditions $u_0$ and $\widetilde
  u_0$ in~$L^\infty(\R)$. Suppose, moreover, that~$\widetilde u_0$ is non-decreasing
  and $u_0-\widetilde u_0\in
  L^1(\R)$.  Then there exists a constant~$C=C(\alpha)>0$ such that for all~$p\in [1,{+\infty}]$ 
  and all $t>0$
\begin{equation}\label{ineq:stab}
\|u(t)-\widetilde u(t)\|_p\leq Ct^{-\frac1\alpha(1-\frac1p)}\|u_0-\widetilde u_0\|_1.
\end{equation}
\end{theorem}
\begin{rem}
\begin{enumerate}
\item It is worth mentioning that this estimate is sharper than the
one obtained by interpolating the $L^1$-contraction principle and $L^\infty$-bounds on the solutions.
\item Mention also that this result holds true for~$\alpha \in (1,2]$ 
without additional~$BV$-assumption on~$u_0$. Consequently, as an immediate 
corollary of~\eqref{asymp:raref} and~\eqref{ineq:stab}, 
one can slightly complete the results from~\cite{kmx}. 
More precisely, let~$\alpha \in (1,2]$,~$u_0 \in L^\infty(\R)$
satisfying~\eqref{ini as} and~$u$ be the solution
to~\eqref{e1.1}--\eqref{e1.2}. Then for every $p\in
(\frac{3-\alpha}{\alpha-1}, +\infty]$ there exists a constant $C>0$ such that for
all $t>0$
$$
\|u(t)-w^R(t)\|_{p}\leq
Ct^{-\frac12[\alpha-1-\frac{3-\alpha}{p}]}\log (2+t)+Ct^{-\frac1\alpha(1-\frac1p)},
$$
even if~$u_0 \notin BV(\R)$.
\end{enumerate}
\end{rem}

In the case $\alpha<1$, the linear part of the fractal Burgers equation
dominates the nonlinear one for large times. In the case $\alpha=1$, 
both parts are balanced; indeed, self-similar solutions exist. Let us
be more precise now.

For $\alpha<1$, the Duhamel principle (see equation
\rf{duh:eps} below) shows that the nonlinearity in 
equation
\rf{e1.1} is
negligible in the asymptotic expansion of solutions.
\begin{theorem}{\bf (Asymptotic behavior as the linear part)} \label{thm:a<1}\\
  Let $0<\alpha<1$ and~$u_0 \in L^\infty(\R)$ satisfying~\eqref{ini as}. 
  Let~$u$ be the entropy solution to
  \rf{e1.1}--\rf{e1.2}.  Denote by
  $\{S_\alpha(t)\}_{t>0}$ the semi-group of linear operators whose infinitesimal generator is~$-\Da$. 
  Consider the initial condition 
\begin{equation}
U_0(x) \equiv  \left\{
\begin{aligned}
u_-\,,\quad x<0, \\
u_+\,,\quad x>0.
\end{aligned}\right. \label{ini:1}
\end{equation}
Then, there exists a constant $C=C(\alpha)>0$ such
  that for all~$p\in
  \big(\frac{1}{1-\alpha},+\infty\big]$ and all~$t>0$,
\begin{equation}\label{in:as:1}
\begin{split}
\|u(t)-S_\alpha(t)U_0\|_p
\leq &
Ct^{-\frac1\alpha(1-\frac1p)} \|u_0-U_0\|_1\\
&+C(u_+-u_-)\max \{|u_+|, |u_-|\}\; t^{1-\frac1\alpha(1-\frac1p)}.
\end{split}
\end{equation} 
\end{theorem}
\begin{rem}\label{rem:p:alpha}
\begin{enumerate}
\item It follows from the proof of Theorem \ref{thm:a<1} that inequality \rf{in:as:1}
is valid for every $p\in [1,+\infty]$. However, its right-hand-side decays only for
$p\in \big(\frac{1}{1-\alpha},+\infty\big]$.
\item Let us recall here the formula $S_\alpha(t)U_0 = p_\alpha(t)*U_0$ where 
$p_\alpha = p_\alpha(x,t)$ denotes the fundamental solution of the equation  $u_t+\Da u=0$ 
({\it cf.} the beginning of Section \ref{sect:regular} for its properties).
Hence, changing  variables in the convolution   $p_\alpha(t)*U_0$, one can write 
the asymptotic term in \rf{in:as:1} in the self-similar form
$(S_\alpha(t)U_0)(x) = H_\alpha(xt^{-1/\alpha})$ where $H_\alpha(x) =(p_\alpha(1)*U_0)(x)$ 
is a smooth and non-decreasing function satisfying 
$
\lim_{x\to\pm \infty} H_\alpha(x)=u_\pm$
and 
$\partial_x H_\alpha(x) =(u_+-u_-)p_\alpha(x,1).
$
\end{enumerate}
\end{rem}

In the case $\alpha=1$, we use the uniqueness result from \cite{A07}
combined with a standard scaling technique to show that equation
\rf{e1.1} has self-similar solutions. 
In Section~\ref{sec:convergence}, we recall this well-known reasoning which leads to the proof of the following theorem.
\begin{theorem} {\bf (Existence of self-similar solutions)} \label{thm:a=1:self}\\ 
Assume  $\alpha=1$.
The unique entropy solution $U$ of the initial value problem
\eqref{e1.1}--\eqref{e1.2} with the initial condition~\eqref{ini:1}
is self-similar, i.e. it has the form~$
U(x,t)=U\left(\frac{x}{t},1\right)
$
for all~$x\in \R$ and all~$t>0$.
\end{theorem}
Our second main convergence result states that the self-similar
solution describes the large time asymptotics of other
solutions to \rf{e1.1}--\rf{e1.2}.
\begin{theorem}{\bf (Asymptotic behavior as the self-similar solution)} \label{thm:a=1}\\
  Let $\alpha=1$ and~$u_0 \in L^\infty(\R)$ satisfying~\eqref{ini as}.  
  Let $u$ be the entropy solution to problem
  \rf{e1.1}--\rf{e1.2}.  Denote by
  $U$ the self-similar solution from Theorem
  \ref{thm:a=1:self}.  Then there exists a
  constant $C=C(\alpha)>0$ such that for all~$p\in [1, +\infty]$ and all~$t>0$,
\begin{align}\label{estimate-remark-optimal}
\|u(t)-U(t)\|_{p}\leq Ct^{-\left(1-\frac{1}{p}\right)}\|u_0-U_0\|_1.
\end{align} 
\end{theorem}


\subsection{Properties of self-similar solutions}

Let us complete the result stated in Theorem \ref{thm:a=1} by listing
main qualitative properties of the profile~$U(1)$.
\begin{theorem}{\bf (Qualitative properties of the self-similar profile)}\label{thm:qualitative}\\
  The self-similar solution  from
  Theorem~\ref{thm:a=1:self} enjoys the following properties:
\begin{itemize}
\item[p1.] {\rm (Regularity)} The function
$U(1)=U(\cdot,1)$ is Lipschitz-continuous.

\item[p2.] {\rm (Monotonicity and limits)}
$U(1)$ is increasing and satisfies 
$$ 
\lim_{x \rightarrow \pm \infty} U(x,1)=u_\pm.
$$

\item[p3.] {\rm (Symmetry)} 
For all~$y \in \R$, we have
$$
U\left(\overline{c}+y,1\right)=2 \overline{c}-U\left(\overline{c}-y,1\right)
\quad \mbox{where}\quad 
\overline{c}\equiv\frac{u_-+u_+}{2}.
$$

\item[p4.] {\rm (Convex/concave)} $U(1)$ is convex (resp. concave)
  on~$(-\infty,\overline{c}]$ (resp. on $[\overline{c},+\infty)$).

\item[p5.] {\rm (Decay at infinity)} We have
$$
U_x(x,1) \sim \frac{u_+-u_-}{2 \pi^{2}} \, |x|^{-2}
\quad \mbox{as} \quad 
|x| \rightarrow +\infty.
$$
\end{itemize}
\end{theorem}
Actually, the profile $U(1)$ is expected to be $C^\infty_b$ or
analytic, due to recent regularity results \cite{kns,cc08,mw08} for
the critical fractal Burgers equation with~$\alpha=1$. It was shown
that the solution is smooth whenever~$u_0$ is either periodic or from
$L^2(\R)$ or from a critical Besov space. Unfortunately, we do not
know if those results can be adapted to any initial condition from
$L^\infty(\R)$.

Property~p3 implies that~$U(x(t),t)$ is a constant equal
to~$\overline{c}$ along the characteristic $x(t)=\overline{c} t$, with
the symmetry
$$
U\left(\overline{c}t+y,t\right)=2 \overline{c}-U\left(\overline{c}t-y,t\right)
$$  
for all~$t >0$ and~$y \in \R$. Thus, the real number~$\overline{c}$ can  be
interpreted as a mean celerity of the profile $U(t)$, which is the
same mean celerity as for the rarefaction wave in~\eqref{rarefaction}.

In property p5, we obtain the decay at infinity which is the same as
for the fundamental solutions $p_1(x,t)=t^{-1}p_1\left(x t^{-1},1\right)$ of the
linear equation~$u_t+\Lambda^1 u=0$, given by the explicit formula
\begin{equation}\label{cauchy}
p_1(x,1)=\frac{2}{1+4 \pi^2 x^2}.
\end{equation}  
Following the terminology introduced in \cite{BK}, one may say that 
property~p5 expresses a far field asymptotics and is somewhere in
relation with the results in~\cite{BK} for fractal conservation laws
with~$\alpha \in (1,2)$, where the Duhamel principle plays a crucial
role. This principle is less convenient in the critical
case~$\alpha=1$, and our proof of~p5 does not use it.

Finally, if~$u_-=0$ and~$u_+-u_-=1$, property p2 means that~$U(1)$ is
the cumulative distribution function of some probability
law~$\mathcal{L}$ with density $U_x(1)$. Property~p3 ensures
that~$\mathcal{L}$ is symmetrically distributed around its median
$\overline{c}$; notice that any random variable with law~$\mathcal{L}$
has no expectation, because of property~p5.  Properties~p4-p5 make
precise that the density of~$\mathcal{L}$ decays around~$\overline{c}$
with the same rate at infinity as for the Cauchy law with
density~$p_1(x,1)$.

The probability distributions of both  laws around their
respective medians can be compared as follows.
\begin{theorem}\label{thm:comp:cauchy}{\bf (Comparison with the Cauchy law)}\\
  Let~$\mathcal{L}$ be the probability law with density $U_x(1)$,
  where $U$ is the self-similar solution defined in
  Theorem~\ref{thm:a=1:self}, with~$u_-=0$ and~$u_+=1$. Let~$X$
  (resp. $Y$) be a real random variable on some probability
  space~$(\Omega,\mathcal{A},\mathbb{P})$ with law $\mathcal{L}$
  (resp.  the Cauchy law~\eqref{cauchy} (with zero
  median)).  Then, we have for
  all~$r>0$
$$
\mathbb{P}(|X-\overline{c}|<r) <\mathbb{P}(|Y-0|<r)
$$
where $\overline{c}$ denotes the median of~$X$.
\end{theorem}
\begin{rem}
  More can be said in order to compare random variables 
$X-\overline{c}$ and $Y$.
  Indeed, their cumulative distribution functions satisfy
  $F_{X-\overline{c}}(x)=F_Y(x)-g(x)$ where $g$ is an explicit positive 
  function (on the positive axis) depending the self-similar solution of \eqref{e1.1} (see
  equation~\eqref{def:difference:cauchy}).
\end{rem}

\subsection{Probabilistic interpretation of results for $\alpha \in (0,2]$}\label{sec-proba} 

To summarize, let us emphasize the probabilistic meaning of the
complete asymptotic study of the fractal Burgers equation we have now
in hands. We have already mentioned that the solution $u$ of
\rf{e1.1}--\rf{e1.2} supplemented with the initial datum of the form
\rf{u0} with $c=0$ and with a probability measure $m$ on $\R$ is the
cumulative distribution function for every $t\geq 0$. This family of
probabilities defined by problem \rf{e1.1}-\rf{e1.2} behaves asymptotically when
$t\to +\infty$ as
\begin{itemize}
\item the uniform distribution on the interval $[0,t]$ if $1<\alpha\leq 2$
(see the result from \cite{kmx} recalled  in inequality \rf{asymp:raref} above);
\item the family of laws $\{\mathcal{L}_t\}_{t\geq 0}$ constructed in Theorem 
\ref{thm:a=1:self} if $\alpha =1$ (see Theorem~\ref{thm:a=1});
\item the symmetric $\alpha$-stable laws $p_\alpha(t)$   if $0<\alpha<1$
({\it cf.}~Theorem \ref{thm:a<1} and Remark \ref{rem:p:alpha}).
\end{itemize}

\subsection{Organization of the article.}
The remainder of this paper is organized as follows.  In the next
section, we recall the notion of entropy solutions to
\rf{e1.1}-\rf{e1.2} with $\alpha\in (0,1]$. Results on the regularized
equation ({\it i.e.} equation \rf{e1.1} with an additional term
$-\varepsilon u_{xx}$ on the left-hand-side) are gathered in
Section~\ref{sect:regular}.  The convergence of solutions as
$\varepsilon \to 0$ to the regularized problem is discussed in
Section~\ref{sec:convergence}. The main asymptotic results
for~\rf{e1.1}-\rf{e1.2} are proved in Section~\ref{sect:conv} by
passage to the limit as $\varepsilon$ goes to
zero. Section~\ref{sec:qualitative} is devoted to the qualitative
study of the self-similar profile for~$\alpha=1$. For the reader's convenience, sketches of proofs of 
a key estimate 
from~\cite{kmx} and Theorem \ref{thm:convg} are given in appendices; the technical
lemmata are also gathered in appendices.


\section{Entropy solutions for $0<\alpha\leq 1$}\label{sec:entropy}

\subsection{L\'evy-Khintchine's representation of $\Lambda^\alpha$}

It is well-known  that the operator
$\Lambda^\alpha=(-\partial^2/\partial x^2)^{\alpha/2}$ for $\alpha \in
(0,2)$ has an integral representation: for every Schwartz function
$\varphi \in \mathcal{S}(\R)$ and each $r>0$, we have
\begin{equation}\label{eqn:LK}
\Lambda^\alpha \varphi =\Lambda_r^{(\alpha)} \varphi +\Lambda_r^{(0)} \varphi,
\end{equation}  
where the integro-differential operators $\Lambda_r^{(\alpha)}$ and
$\Lambda_r^{(0)}$ are defined by
\begin{eqnarray}
  \Lambda_r^{(\alpha)} \varphi (x) & \equiv & 
  -G_\alpha \int_{|z| \leq r} \frac{\varphi(x+z)-\varphi(x)-\varphi_x(x) z}{|z|^{1+\alpha}} \; dz , \label{LK1}
  \\
  \Lambda_r^{(0)} \varphi (x) & \equiv & -G_\alpha \int_{|z| > r} \frac{\varphi(x+z)-\varphi(x)}{|z|^{1+\alpha}} \; dz, 
\label{LK2}
\end{eqnarray}
where $G_\alpha \equiv \frac{\alpha \Gamma\left(\frac{1+\alpha}{2} \right)}{2
  \pi^{\frac{1}{2}+\alpha}\Gamma\left(1-\frac{\alpha}{2}\right)}>0$
and $\Gamma$ is Euler's function.
On the basis of this formula, we can extend the domain of definition
of $\Lambda^\alpha$ and consider $\Lambda_r^{(0)}$ and
$\Lambda_r^{(\alpha)}$ as the operators
\begin{equation*}
  \Lambda_r^{(0)}:C_b(\R) \rightarrow C_b(\R) \mbox{ and } \Lambda_r^{(\alpha)}:C_b^2(\R) \rightarrow C_b(\R);
\end{equation*}
hence, $\Lambda^\alpha:C_b^2(\R) \rightarrow C_b(\R).$

Let us recall some properties on these operators.  First, the
so-called Kato inequality can be generalized to $\Lambda^\alpha$ for
each $\alpha\in (0,2]$: let $\eta \in C^2(\R)$ be convex and $\varphi
\in C_b^2(\R)$, then
\begin{equation}\label{ineq:kato}
\Lambda^\alpha \eta(u) \leq \eta'(u) \Lambda^\alpha u.
\end{equation}
Note that for $\alpha=2$ we have
$$
-(\eta(u))_{xx}=-\eta''(u) u_x^2-\eta'(u) u_{xx} \leq -\eta'(u) u_{xx} \quad \mbox{since $\eta'' \geq 0$}.
$$
If $\alpha \in (0,2)$, inequality \eqref{ineq:kato} is the direct
consequence of the integral representation \eqref{eqn:LK}--\rf{LK2}
and of the following inequalities
\begin{equation}\label{ineq:kator}
  \Lambda_r^{(0)} \eta(u) \leq \eta'(u) \Lambda_r^{(0)} u 
  \quad \mbox{and} \quad \Lambda_r^{(\alpha)} \eta(u) \leq \eta'(u) \Lambda_r^{(\alpha)} u,
\end{equation}
resulting from the convexity of the function $\eta$.

Finally, these operators satisfy the integration by parts formula: for
all $u \in C_b^2(\R)$ and $\varphi \in \mathcal{D}(\R)$, we have
\begin{equation}\label{ipp}
\int_\R  \varphi \Lambda u\,dx=\int_\R u \Lambda \varphi \,dx, 
\end{equation}
where $\Lambda\in
\{\Lambda_r^{(0)},\Lambda_r^{(\alpha)},\Lambda^\alpha\}$ for every
$\alpha\in (0,2]$ and all $r>0$. Notice that~$\Lambda \varphi \in
L^1(\R)$, since it is obvious from~\eqref{LK1}-\eqref{LK2}
that~$\Lambda_r^{(\alpha)}:W^{2,1}(\R) \rightarrow L^1(\R)$
and~$\Lambda_r^{(0)}:L^1(\R) \rightarrow L^1(\R)$.

Detailed proofs of all these properties are based on the
representation \eqref{eqn:LK}--\rf{LK2} and are written {\it e.g.}  in
\cite{A07}.


\subsection{Existence and uniqueness of entropy solutions}

It was shown in \cite{adv07} (see also \cite{kns}) that solutions of
the initial value problem for the fractal conservation law
\begin{eqnarray}\label{eg1.1}
&&u_t+\Da u +(f(u))_x=0, \quad x\in\R,\; t>0,\\
\label{eg1.2} &&u(0,x)=u_0(x), 
\end{eqnarray}
where $f:\R \rightarrow \R$ is locally Lipschitz-continuous, can
become discontinuous in finite time if $0<\alpha<1$. Hence, in order
to deal with discontinuous solutions, the notion of entropy solution
in the sense of Kruzhkov was extended in~\cite{A07} to fractal
conservation laws \rf{eg1.1}--\rf{eg1.2} (see also~\cite{KaUlpr} for
the recent generalization to L\'evy mixed hyperbolic/parabolic
equations). Here, the crucial role is played by the
L\'evy-Khintchine's representation \eqref{eqn:LK}--\rf{LK2} of the
operator $\Da$.
\begin{defin}\label{def:entropy}
  Let $0<\alpha\leq 1$ and $u_0\in L^\infty(\R)$. A function $u \in
  L^\infty(\R \times (0,+\infty))$ is an \emph{entropy solution} to
  \rf{eg1.1}--\rf{eg1.2} if for all $\varphi \in \mathcal{D}(\R \times
  [0,+\infty))$, $\varphi \ge 0$, $\eta \in C^2(\R)$ convex, $\phi:\R
  \rightarrow \R$ such that $\phi'=\eta' f'$, and $r>0$, we have
\begin{equation*}
\begin{split}
\int_{\R} \int_0^{+\infty} \Big(\eta(u) \varphi_t +\phi(u) \varphi_x
-\eta(u) \Lambda_r^{(\alpha)} \varphi 
&-\varphi \eta'(u)  \; \Lambda_r^{(0)} u \Big)\,dxdt \\
&+\int_{\R} \eta(u_0(x)) \varphi(x,0)\,dx \geq 0.
\end{split}
\end{equation*} 
\end{defin}
Note that, due to formula \rf{LK2}, the quantity $\Lambda_r^{(0)} u$
in the above inequality is well-defined for any bounded function $u$.

The notion of entropy solutions allows us to solve the fractal
Burgers equation for the range of exponent~$\alpha \in (0,1]$.
\begin{theorem}[\cite{A07}]
\label{thm:entropy}
Assume that $0<\alpha\leq 1$ and $u_0\in L^\infty(\R)$.  There exists
a unique entropy solution $u$ to problem
\rf{eg1.1}--\rf{eg1.2}.  This solution $u$ belongs to
$C([0,{+\infty});L^1_{loc}(\R))$ and satisfies $u(0)=u_0$. Moreover, we
have the following maximum principle: $\mbox{\emph{ess\,inf}} \, u_0
\leq u \leq \mbox{\emph{ess\,sup}} \, u_0$.
\end{theorem}
If $\alpha\in (1,2]$, all solutions to \rf{eg1.1}--\rf{eg1.2} with
bounded initial conditions are smooth and global-in-time (see
\cite{dgv, kns,MYZ}).  On the other hand, the occurrence of
discontinuities in finite time of entropy solutions to
\rf{eg1.1}--\rf{eg1.2} with $\alpha=1$ seems to be unclear. As
mentioned in the introduction, regularity results have recently been
obtained \cite{kns,cc08,mw08} for a large class of initial conditions
which, unfortunately, does not include general $L^\infty$-initial
data. Nevertheless, Theorem \ref{thm:entropy} provides the existence
and the uniqueness of a global-in-time entropy solution even for the
critical case~$\alpha=1$.

\section{Regularized problem}\label{sect:regular}

In this section, we gather properties of solutions to the Cauchy
problem for the regularized fractal Burgers equation with
$\varepsilon>0$
\begin{eqnarray}
&&\ue_t+\Da \ue -\varepsilon \ue_{xx} +\ue\ue_x=0, \quad x\in\R,\; t>0,\label{ue1}\\ 
&&\ue(x,0)=u_0(x).\label{ue2}
\end{eqnarray}
Our purpose is to derive asymptotic stability estimates of a solution
$\ue=\ue(x,t)$ (uniform in $\varepsilon$) that will be valid
for~\eqref{e1.1}--\eqref{e1.2} after passing to the limit
$\varepsilon\to 0$.  Most of the results of this section are based on a key estimate 
from~\cite{kmx}; unfortunately, this estimate is not explicitely stated as a lemma in~\cite{kmx}. 
Hence, for the sake of completeness, we have recalled this key estimate 
in Lemma~\ref{KMX} in Appendix~\ref{appendix-KMX}
as well as the main lines of its proof. 

Below, we will use the following integral formulation of the initial value
problem~\eqref{ue1}-\eqref{ue2}
\begin{equation}\label{duh:eps}
\ue(t)=S^\varepsilon_\alpha(t)u_0-
\int_0^t S^\varepsilon_\alpha(t-\tau)\ue(\tau)\ue_x(\tau)\,d\tau, 
\end{equation}
where~$\{S_\alpha^\varepsilon(t)\}_{t>0}$ is the semi-group
generated by~$-\Da +\varepsilon \partial_x^2$.

If, for each $\alpha \in (0,2]$, the function
$p_\alpha$ denotes the
fundamental solution of the linear equation $u_t+\Da u=0$, then
\begin{equation}\label{Sae}
S_\alpha^\varepsilon(t) u_0= p_\alpha(t) \ast p_2(\varepsilon t) \ast u_0.
\end{equation}
It is well-known that $p_\alpha=p_\alpha(x,t)$ can be represented via the Fourier transform 
(w.r.t. the~$x$-variable)
$ \widehat p_\alpha(\xi,t)=e^{-t|\xi|^\alpha}$.  In particular,
\begin{equation}\label{homogeneity}
 p_\alpha(x,t)=t^{-\frac1\alpha}P_\alpha(xt^{-\frac1\alpha}),
\end{equation}
where $P_\alpha$ is the inverse Fourier transform of $e^{-
  |\xi|^\alpha}$.  For every $\alpha\in (0,2]$ the function $P_\alpha$
is smooth, non-negative, $\int_\R P_\alpha(y)\,dy=1$, and satisfies the
estimates (optimal for $\alpha\neq 2$)
\begin{equation}
\label{EP}
0<P_\alpha(x)\le C(1+|x|)^{-(\alpha+1)} \;\; \hbox{and}\;\; 
|\partial_x P_\alpha|\le C(1+|x|)^{-(\alpha+2)}
\end{equation}
for a constant $C$ and all $x\in\R$.
\medskip


One can see that problem \rf{ue1}--\rf{ue2} admits a unique
global-in-time smooth solution. 
\begin{theorem}[\cite{dgv}]\label{th2.1}
Let $\alpha\in (0,\,2]$,  $\varepsilon >0$ and $u_0\in L^\infty(\R)$.
There exists a unique solution $\ue$ to problem~\rf{ue1}--\rf{ue2} in
the following sense: 
\begin{itemize}
\item
$\ue\in C_b(\mathbb{R} \times (0,+\infty)) \cap C^\infty_b(\mathbb{R} \times (a,+\infty))$ for all $a>0$,
\item
$\ue$ satisfies equation (\ref{ue1}) on $\mathbb{R} \times (0,+\infty)$,
\item
$\lim_{t \rightarrow 0} \ue(t) = u_0$ in
$L^\infty(\mathbb{R})$ weak-$\ast$ and in $L^p_{loc}(\mathbb{R})$
for all~$p\in [1,{+\infty})$.
\end{itemize}
Moreover, the following maximum principle holds true: 
\begin{equation}\label{est:inf}
\mbox{\emph{ess\,inf}} \, u_0 \leq \ue \leq \mbox{\emph{ess\,sup}} 
\, u_0.
\end{equation}
\end{theorem}
\begin{proof}
Here, the results from \cite{dgv} can be easily  modified in order to get
 the existence and the regularity of solutions to \rf{ue1}--\rf{ue2} 
with $\varepsilon>0$.
\end{proof}

Here are some elementary properties (comparison principle,~$L^1$-contraction principle and 
non-increase of the~$BV$-semi-norm) of fractal conservation laws that will be needed.
\begin{prop}[\cite{dgv}]\label{rem-droniou}
Let~$\varepsilon>0$ and~$u^\varepsilon$ and~$\widetilde{u^\varepsilon}$ be solutions to~\eqref{ue1}--\eqref{ue2}
with respective initial data~$u_0$ and~$\widetilde{u_0}$ in~$L^\infty(\R)$. Then:
\begin{itemize}
\item if~$u_0 \leq \widetilde{u_0}$ then 
$
u^\varepsilon \leq \widetilde{u^\varepsilon}
$,
\item if~$u_0-\widetilde{u_0} \in L^1(\R)$ then 
$
\|u^\varepsilon-\widetilde{u^\varepsilon}\|_{L^\infty(0,+\infty,L^1)} \leq \|u_0-\widetilde{u_0}\|_1
$,
\item if~$u_0 \in BV(\R$) then 
$
\|u_x^\varepsilon(t)\|_{L^\infty(0,+\infty,L^1)} \leq |u_0|_{BV}
$
\end{itemize}
where~$\|\cdot\|_{L^\infty(0,+\infty,L^1)}$ and~$|\cdot|_{BV}$ denote respectively the~norm in
~$L^\infty(0,+\infty,L^1(\R))$ and the semi-norm in~$BV(\R)$.
\end{prop}

\begin{proof}[Sketch of the proof]
As explained in~\cite[Remarks 1.2 \& 6.2]{dgv}, these 
properties are immediate consequences of the splitting method developped in~\cite{dgv} and
the facts that both the
hyperbolic equation~$u_t+u u_x=0$ and the fractal equation~$u_t+\Lambda^\alpha u-\varepsilon u_{xx}=0$
satisfy these properties. 
\end{proof}

The next proposition provides an estimate on the gradient of~$u^\varepsilon$. 
\begin{prop}\label{th:ux}
Let $0<\alpha\leq 1$ and~$u_0 \in L^\infty(\R)$ be non-decreasing.
For each~$\varepsilon>0$, denote by $u^\varepsilon$
the solution to~\rf{ue1}--\rf{ue2}.
Then:
\begin{itemize}
\item $\ue_x(x,t)\geq 0$ for all $x\in\R$ and $t>0$,

\item there exists a constant $C=C(\alpha)>0$ 
such that for all~$\varepsilon>0$,~$p\in [1,{+\infty}]$ and~$t>0$,
\begin{equation}\label{ux dec}
\|\ue_x(t)\|_p\leq  Ct^{-\frac1\alpha(1-\frac1p)} |u_0|_{BV}.
\end{equation}
\end{itemize}
 \end{prop}

\begin{proof}
For any fixed real~$h$, the function~$u^\varepsilon(\cdot +h,\cdot)$ is the
solution to \eqref{ue1}--\eqref{ue2} with the initial datum
$u_0(\cdot +h)$. Consequently, for non-decreasing $u_0$ and for
$h>0$, the inequality $u_0(\cdot+h)\geq u_0(\cdot)$  and the comparison principle 
imply
$u^\varepsilon(\cdot +h,\cdot)\geq u^\varepsilon(\cdot,\cdot)$ which
gives $u^\varepsilon_x\geq 0$.  

To show the decay of the~$L^p$-norm, one slightly modifies the arguments from~\cite[Proof of
Lemma 3.1] {kmx}. One 
shall use Lemma~\ref{KMX} with~$v \equiv u_x^\varepsilon$. It is clear that~$v$ 
satisfies the required regularity: for all~$a>0$
$$
v \in C^\infty_b(\R \times (a,+\infty)) \cap L^\infty(0,+\infty,L^1(\R)),
$$ 
thanks to Proposition~\ref{rem-droniou}
ensuring that
$$
\|v\|_{L^\infty(0,+\infty,L^1)} = \| u_x^\varepsilon \|_{L^\infty(0,+\infty,L^1)} \leq |u_0|_{BV}.
$$
It thus rests to show that~$v$ satisfies~\eqref{assumptionii}.
By interpolation of the inequality above and the $L^\infty$-bound on~$v$ from
Theorem~\ref{th2.1}, one 
sees that for all~$p \in [1,+\infty]$ and all~$t>0$, 
$$
v(t) \in L^p(\R) \quad \mbox{and} \quad v_t(t),\Lambda^\alpha v(t),v_x(t),v_{xx}(t) \in L^\infty(\R). 
$$
Hence, for~$p \in [2,+\infty)$, one can multiply the equation for
$v$
$$
v_t+\Da v -\varepsilon v_{xx} +(\ue\ue_x)_x=0,
$$
by $v^{p-1}$ to obtain after integration: 
\begin{equation}\label{ue:Lp:ineq}
\int_{\R} v_t v^{p-1} \, dx +\int_\R v^{p-1}\Da v\,dx -
\varepsilon \int_\R  v_{xx} v^{p-1}\,dx +\frac{p-1}{p}
\int_\R v^{p+1}\,dx=0;
\end{equation}
here one has used that~
$\lim_{|x| \rightarrow +\infty} v(x,t)=0$ (since~$v(t) \in C^\infty_b(\R) \cap L^1(\R)$) to drop the boundary terms
providing from integration by parts. Integrating again by parts, one sees that 
\begin{equation*}
-\varepsilon \int_\R v_{xx}\Phi(v) \,dx=\varepsilon \int_{\R} v_x^2 \Phi'(v) \, dx \geq 0
\end{equation*} 
for all non-decreasing function~$\Phi \in C^1(\R)$ with~$\Phi(0)=0$;
Choosing~$\Phi(v)=|v|^{p-2} v$, one gets
\begin{equation}\label{Kato:Lap}
-\varepsilon \int_\R v_{xx} |v|^{p-2} \, v \,dx  \ge 0. 
\end{equation} 
We deduce from~\eqref{ue:Lp:ineq},~\eqref{Kato:Lap} and the non-negativity of~$v$ 
that
$$
\int_{\R} v_t |v|^{p-2}v \, dx+\int_\R |v|^{p-2} v \Lambda^\alpha v dx \leq 0
$$
for all~$p \in [2,+\infty)$ and~$t>0$. This is exactly the required differential inequation in~\eqref{assumptionii}.
Lemma~\ref{KMX} thus completes the proof.
\end{proof}
We can now give asymptotic stability estimates uniform in~$\varepsilon$.
\begin{theorem}\label{thm:as-stab} Let $\alpha\in  (0,2]$.
  Consider two initial data $u_0$ 
  and $\widetilde u_0$ in~$L^\infty(\R)$ such that   
  ~$\widetilde u_0$ is non-decreasing and
  $u_0-\widetilde u_0\in L^1(\R)$.  
  For each~$\varepsilon>0$, denote by~$\ue$ and~$\widetilde \ue$ the corresponding 
  solutions to~\rf{ue1}--\rf{ue2}. 
  Then, 
  there exists a constant $C=C(\alpha)>0$ such
  for all~$\varepsilon>0$,~$p\in [1,{+\infty}]$
   and~$t>0$
\begin{equation}\label{stab:est:ep}
\|\ue(t)-\widetilde \ue(t)\|_p\leq Ct^{-\frac1\alpha(1-\frac1p)}\|u_0-\widetilde u_0\|_1 \, .
\end{equation}
\end{theorem}

\Proof 
The proof follows the arguments from~\cite[Proof of
Lemma 3.1] {kmx} by skipping the additional term providing from~$-\varepsilon u^\varepsilon_{xx}$.
That is to say, one uses again Lemma~\ref{KMX} with~$v=u^\varepsilon-\widetilde{u^\varepsilon}$. 
First, the
$L^1$-contraction principle (see Proposition~\ref{rem-droniou}) ensures that~$v$ satisfies the required regularity with
$$
\|\ue-\widetilde \ue\|_{L^\infty(0,+\infty,L^1)} \leq \|u_0-\widetilde{u_0}\|_1.
$$
In particular, once again 
the interpolation of the~$L^1$- and
~$L^\infty$-norms implies that~$v$ is~$L^p$ in space for all time and all~$p \in [1,+\infty]$.
Second, one takes~$p \in [2,+\infty)$ (so that all the integrands below are integrable) 
and one multiplies the difference of 
the equations satisfied by~$u^\varepsilon$ and~$\widetilde{u^\varepsilon}$ by~$|v|^{p-2} v$. One gets
after integration:
\begin{multline}\label{precision}
\int_{\R} v_t |v|^{p-2}v \, dx+\int_{\R} |v|^{p-2} v \Lambda^\alpha v \, dx\\
-\varepsilon \int_\R v_{xx} |v|^{p-2} v  \, dx+\frac{1}{2} \int_{\R} 
\left(v^2+2v \widetilde{u^\varepsilon}\right)_x |v|^{p-2} v \, dx=0.
\end{multline}
The last term of the left-hand side of this equality is non-negative, since integrations by parts give
\begin{align*}
& \int_{\R} 
\left(v^2+2v \widetilde{u^\varepsilon}\right)_x |v|^{p-2} v \, dx\\
 & =  \int_{\R} 2v_x |v|^p \, dx+
\int_{\R} 2\widetilde{u^\varepsilon}v_x |v|^{p-2} v \, dx+
\int_{\R} 2\widetilde{u_x^\varepsilon} |v|^{p}  \, dx, \\
& =  2 \left(1-\frac{1}{p} \right) \int_{\R} 2\widetilde{u_x^\varepsilon} |v|^{p}  \, dx \geq 0
\end{align*}
(once again the boundary terms can be skipped since~$v$ vanishes for large~$x$).
Moreover the third term of~\eqref{precision} is also non-negative by~\eqref{Kato:Lap}. 
One easily deduces the desired inequality~\eqref{assumptionii} and completes the proof by Lemma~\ref{KMX}.  
\endProof

\begin{theorem} \label{lem:lin:as} Let $0<\alpha< 1$ and~$u_0 \in L^\infty(\R)$ be non-decreasing. 
For each~$\varepsilon>0$, denote by~$\ue$ the solution to~\rf{ue1}--\rf{ue2}.
Then, there exists $C=C(\alpha)>0$ such that for all~$\varepsilon>0$,
~$p\in [1,{+\infty}]$ and~$t>0$
\begin{equation*}
\|\ue(t)-S^\varepsilon_\alpha(t)u_0\|_p
\leq C\|u_0\|_\infty |u_0|_{BV} t^{1-\frac1\alpha(1-\frac1p)}
\end{equation*}
(where~$\{S^\varepsilon_\alpha(t)\}_{t>0}$ is generated by~$-\Da+\varepsilon \partial^2_x$).  
\end{theorem}
\begin{proof}
  Using the integral equation \rf{duh:eps} we immediately obtain
\begin{equation}\label{duh:p}
  \|\ue(t)-S^\varepsilon_\alpha(t)u_0\|_p
  \leq\int_0^t \left\|S^\varepsilon_\alpha(t-\tau)\ue(\tau)\ue_x(\tau)\right\|_p\;d\tau.
\end{equation}
Now, we estimate the integral in the right-hand side of \rf{duh:p}
using the $L^p$-decay of the semi-group $S^\varepsilon_\alpha(t)$ as
well as inequalities \rf{est:inf} and \rf{ux dec}. Indeed, it follows
from \rf{homogeneity}-\rf{EP} that
$$
\|p_2(\varepsilon t)\|_1=1\quad \mbox{and}\quad
\|p_\alpha(t)\|_r=t^{-\frac1\alpha(1-\frac1r)}\|p_\alpha(1)\|_r
$$ 
for every $r\in [1,+\infty]$. Hence, by the Young inequality for the
convolution and inequalities \rf{est:inf}, \rf{ux dec}, we obtain
\begin{equation}\label{Sea:est}
\begin{split}
\|S^\varepsilon_\alpha(t-\tau)&\ue(\tau)\ue_x(\tau)\|_p\\
&\leq
\|p_\alpha(t-\tau)*(\ue(\tau)\ue_x(\tau))\|_p,\\
&\leq C(t-\tau)^{-\frac1\alpha(\frac1q-\frac1p)}
\|\ue(\tau)\|_\infty\|\ue_x(\tau)\|_q,\\
&\leq C(t-\tau)^{-\frac1\alpha(\frac1q-\frac1p)}
\|u_0\|_\infty|u_0|_{BV} \tau^{-\frac1\alpha(1-\frac1q)},
\end{split}
\end{equation}
for all $1\leq q\leq p\leq {+\infty}$, $t>0$, $\tau \in (0,t)$, where the constant~$C$
only depends on~$\max_{r \in [1,+\infty]} \|p_\alpha(1)\|_r$ and the constant in~\eqref{ux dec}.

Next, we decompose the integral on the right-hand side of \rf{duh:p}
as follows
$\int_0^t...\;d\tau=\int_0^{t/2}...\;d\tau+\int_{t/2}^t...\;d\tau$ and
we bound both integrands by using inequality \rf{Sea:est} either with
$q=1$ or with $q=p$. This leads to the following inequality
\begin{equation}
\begin{split}
\|&\ue(t)-S^\varepsilon_\alpha(t)u_0\|_p\\
&\leq C\|u_0\|_\infty |u_0|_{BV} \Bigg(
\int_0^{t/2}(t-\tau)^{-\frac1\alpha(1-\frac1p)}\;d\tau
+ \int_{t/2}^t \tau^{-\frac1\alpha(1-\frac1p)}\; d\tau\Bigg),\\
&=C \|u_0\|_\infty |u_0|_{BV} \frac{2^{\beta}-1}{\beta 2^{\beta-1}}\,t^{\beta},
\end{split}
\end{equation}
where~$\beta \equiv 1-\frac{1}{\alpha} \left(1-\frac{1}{p} \right)$. 
It is readily seen that~$\beta \in \R \rightarrow \frac{2^{\beta}-1}{\beta 2^{\beta-1}}$ is positive and continuous
and that~$p \in [1,+\infty) \rightarrow 1-\frac{1}{\alpha} \left(1-\frac{1}{p} \right)$ is bounded. 
This completes the proof
of Theorem~\ref{lem:lin:as}.
\end{proof}


\section{Entropy solution: parabolic approximation and self-similarity}
\label{sec:convergence}

In this section, we 
state the result on the convergence, as $\varepsilon \to 0$,
of solutions $\ue$ of \rf{ue1}--\rf{ue2}
 toward the entropy solution $u$ of
\rf{e1.1}--\rf{e1.2}. We also prove Theorem~\ref{thm:a=1:self} about
self-similar entropy solutions in the case $\alpha=1$. 

Together with the general fractal conservation law
\eqref{eg1.1}--\eqref{eg1.2}, we study the associated regularized
problem
\begin{eqnarray}
&&\ue_t+\Da \ue -\varepsilon \ue_{xx} +(f(\ue))_x=0, \quad x\in\R,\; t>0,\label{uge1}\\ 
&&\ue(x,0)=u_0(x) \label{uge2}
\end{eqnarray}
where $f \in C^\infty(\R)$.  Hence, by results of \cite{dgv} (see also
Theorem \ref{th2.1}), problem~\rf{uge1}-\rf{uge2} admits a unique,
global-in-time, smooth solution $\ue$.

\begin{theorem}\label{thm:convg} 
Let $u_0 \in L^\infty(\R)$. For each~$\varepsilon>0$, let
$\ue$ be the solution to \rf{uge1}--\rf{uge2} and
$u$ be the entropy solution to \rf{eg1.1}--\rf{eg1.2}.  Then,
for every $T>0$, $\ue \rightarrow u$ in $C([0,T];L^1_{loc}(\R))$ as
$\varepsilon \rightarrow 0$.
\end{theorem}

The proof of Theorem~\ref{thm:convg} is given in Appendix~\ref{correction-point2-appendix}.

\begin{rem}\label{rem-non-optimal}
This result actually holds true for only locally Lipschitz-continuous fluxes~$f$.
More generally, multidimensional fractal
conservation laws with source terms $h=h(u,x,t)$ and fluxes
$f=f(u,x,t)$ (see \cite{dri,dgv}) can be considered. 
\end{rem}

\begin{proof}[Proof of Theorem \ref{thm:a=1:self}]
  The existence of the solution $U$ to equation \rf{e1.1} with
  $\alpha=1$ supplemented with the initial condition \rf{ini:1} is
  provided by Theorem~\ref{thm:entropy}.  To obtain the self-similar
  form of $U$, we follow a standard argument based on the uniqueness
  result from Theorem~\ref{thm:entropy}. Observe that if $U$ is the
  solution to \rf{e1.1}, the rescaled function $
  U^\lambda(x,t)=U(\lambda x, \lambda t) $ is the solution for every
  $\lambda>0$, too.  Since, the initial datum \rf{ini:1} is invariant
  under the rescaling $ U^\lambda_0(x)=U_0(\lambda x)$, by the
  uniqueness, we obtain that for all~$\lambda>0$, $U(x,t)= U(\lambda
  x, \lambda t)$ for a.e. $(x,t)\in \R \times
  (0,+\infty)$. 
\end{proof}


\section{Passage to the limit $\varepsilon\to 0$ and asymptotic study}
\label{sect:conv}

In this section, we prove Theorems~\ref{thm:as-stab-fb}, \ref{thm:a<1} and \ref{thm:a=1}.

\begin{proof}[Proof of Theorem~\ref{thm:as-stab-fb}]
  Denote by $u^\varepsilon$ and $\widetilde u^\varepsilon$ the
  solutions to the regularized equation \rf{ue1} with the initial
  conditions $u_0$ and $\widetilde u_0$.  By Theorem~\ref{thm:convg}
  and the maximum principle \eqref{est:inf}, we know that $\lim_{\varepsilon
    \rightarrow 0} \ue(t)=u(t)$ and $\lim_{\varepsilon \rightarrow 0}
  \widetilde\ue(t)=\widetilde u(t)$ in $L^p_{loc}(\R)$ for every $p
  \in [1,{+\infty})$ and in $L^\infty(\R)$ weak-$\ast$.  Hence, for each
  $R>0$ and $p \in [1,{+\infty}]$, using Theorem \ref{thm:as-stab} we
  have
\begin{equation*}
\begin{split}
\|u(t)-\widetilde u(t)\|_{L^p((-R,R))} 
&\leq \liminf_{\varepsilon \rightarrow 0} 
\|\ue(t)-\widetilde u^\varepsilon(t)\|_{L^p((-R,R))} \\
&\leq Ct^{-\frac1\alpha(1-\frac1p)}\|u_0-\widetilde u_0\|_1.
\end{split}
\end{equation*}
Since $R>0$ is arbitrary and the right-hand side of this inequality does 
not depend on $R$, we complete the proof of inequality \rf{ineq:stab}.
\end{proof}
\begin{proof}[Proof of Theorem \ref{thm:a<1}]
  In view of Theorem \ref{thm:as-stab-fb}, it suffices to show the following inequality
\begin{equation*}
\|\widetilde u(t)-S_\alpha(t) u_0\|_p 
\leq C\|U_0\|_\infty |U_0|_{BV} t^{1-\frac1\alpha(1-\frac1p)},
\end{equation*}
where $\widetilde u$ is the solution to \eqref{e1.1} with $U_0$ as the initial condition.
Notice that $\|U_0\|_\infty= u_+-u_-$ and $|U_0|_{BV}=\max\{|u_+|, |u_-|\}$ in this case.

Here, we argue exactly as in the proof of Theorem
\ref{thm:as-stab-fb}, since we can assume that $\lim_{\varepsilon
\rightarrow 0} \widetilde u^\varepsilon(t)=\widetilde u(t)$ in
$L^p_{loc}(\R)$ for every $p \in [1,{+\infty})$ and in $L^\infty(\R)$
weak-$\ast$.  Moreover, it is well-known that for fixed $t>0$
$$
\lim_{\varepsilon \to 0} S_\alpha^\varepsilon(t) U_0 = 
\lim_{\varepsilon \to 0} p_2(\varepsilon t)*\big(p_\alpha(t)* U_0\big)= 
S_\alpha(t) U_0 \quad \mbox{in} \quad L^p(\R) 
$$ 
for all $p \in [1,{+\infty}]$.
Hence, for every  $R>0$ and $p \in [1,{+\infty}]$, by Theorem  \ref{lem:lin:as},
we obtain
\begin{equation*}
\begin{split}
  \|\widetilde u(t)-S_\alpha(t) U_0\|_{L^p((-R,R))}
  &\leq \liminf_{\varepsilon \rightarrow 0} \|\widetilde u^\varepsilon(t)-S_\alpha^\varepsilon(t) U_0\|_{L^p((-R,R))} \\
  &\leq C\|U_0\|_\infty|U_0|_{BV} t^{1-\frac1\alpha(1-\frac1p)}.
\end{split}
\end{equation*}
The proof is completed by letting $R\to{+\infty}$.
\end{proof}

\begin{proof}[Proof of Theorem \ref{thm:a=1}]
Apply Theorem~\ref{thm:as-stab-fb} with $\alpha=1$ and 
$\widetilde u_0=U_0$.
\end{proof}


\section{Qualitative study of the self-similar profile for~$\alpha=1$}\label{sec:qualitative}

This section is devoted to the proof of Theorems~\ref{thm:qualitative} and \ref{thm:comp:cauchy}. 

\subsection{Proof of properties~p1--p4 from Theorem~\ref{thm:qualitative}}

  The Lipschitz-continuity sta\-ted in p1 is an immediate consequence
  of Proposition~\ref{th:ux} and
  Theorem~\ref{thm:convg}. Indeed,~$U(1)$ is the limit
  in~$L^1_{loc}(\R)$ of~$u^\varepsilon(1)$ as~$\varepsilon \rightarrow
  0$, where~$u^\varepsilon$ is solution
  to~\eqref{ue1}--\eqref{ue2} with~$u_0=U_0$ defined
  in~\eqref{ini:1}. Moreover, by~\eqref{ux dec}, the
  family~$\{u^\varepsilon(1):\varepsilon>0\}$ is
  equi-Lipschitz-continuous, which implies that the limit~$U(1)$ is
  Lipschitz-continuous.

  Before proving properties p2--p4, let us reduce the problem to a
  simpler one. We remark that equation~\eqref{e1.1} is invariant
  under the transformation
 \begin{equation}\label{transformation}
   V(x,t)\equiv U\left(x+\overline{c} t,t\right)-\overline{c} 
   \quad \mbox{where} \quad \overline{c}\equiv \frac{u_-+u_+}{2};
 \end{equation}
 that is to say, if~$U$ is a solution to~\eqref{e1.1}
 with~$U(x,0)=U_0(x)$ defined in~\eqref{ini:1}, then $V$ is a solution
 to~\eqref{e1.1} with the initial datum
 \begin{equation}\label{ini:2}
V(x,0)=V_0(x)\equiv \left\{
\begin{aligned}
v_+,\quad x<0,\\
v_-,\quad x>0,
\end{aligned}\right.
\end{equation}
where~$v_-=-v_+$ and~$v_+\equiv |\overline{c}| \geq 0$. It is clear
that~$U$ satisfies p2--p4, whenever~$V$ enjoys these properties.  In
the sequel, we thus assume without loss of generality that~$u_-=-u_+$
and~$u_+>0$.

It has been shown in~\cite[Lemma 3.1]{adv07} that if~$u_0 \in
L^\infty(\R)$ is non-increasing, odd and convex on~$(0,+\infty)$, then
the solution $u$ of~\eqref{e1.1}-\eqref{e1.2} shares these
properties w.r.t.~$x$, for all $t>0$. The proof is based on a
splitting method and on the fact that the ``odd, concave/convex''
property is conserved by both the hyperbolic equation~$u_t+ u u_x$ and
the fractal equation~$u_t+\Lambda^1 u=0$. The same proof works with
minor modifications to show that if~$u_0$ is non-decreasing, odd and
convex on~$(-\infty,0)$, then these properties are preserved by
problem~\eqref{e1.1}--\eqref{e1.2}. Details are left
to the reader since in that case, no shock can be created by the
Burgers part and the proof is even easier. By the hypothesis
$u_-=-u_+<0$ made above, the initial datum in~\eqref{ini:1} is
non-decreasing, odd and convex on~$(-\infty,0)$. We conclude that so
is the profile~$U(1)$. The proof of properties
p3--p4 is now complete.

What is left to prove is the limit in property p2. By
Theorem~\ref{thm:entropy}, we have~$U(t) \rightarrow U_0$
in~$L^1_{loc}(\R)$ as~$t \rightarrow 0$. In particular, the
convergence holds true a.e.~along a subsequence~$t_n \rightarrow 0$
as~$n \rightarrow {+\infty}$ and there exists~$\pm x_\pm >0$ such
that~$U(x_\pm,t_n) \rightarrow u_\pm$. By the self-similarity of~$U$,
we get~$U\left(\frac{x_\pm}{t_n},1\right) \rightarrow u_\pm $ as~$n \rightarrow
{+\infty}$. Since~$\frac{x_\pm}{t_n} \rightarrow \pm \infty$ and~$U(1)$ is
non-decreasing, we deduce property p2.

\subsection{Some technical lemmata} The last property of
Theorem~\ref{thm:qualitative} is the most difficult part to prove. In
this preparatory subsection, we state and prove technical results that
shall be needed in our reasoning.
\begin{lemma}\label{lem:equivalence}
  Let~$v \in L^\infty(\R)$ be non-negative, even and non-increasing
  on~$(0,+\infty)$. Assume that there exists~$\ell > 0$ such that for
  all~$x_0 > 1/2$,
\begin{equation}\label{eq:tech:decay}
\lim_{n \rightarrow {+\infty}} n^{-1} \int_{n (x_0-1/2)}^{n (x_0 +1/2)} y^2 v(y) dy=\ell.
\end{equation} 
Then, we have~$y^2 v(y) \to_{|y| \rightarrow {+\infty}} \ell$.
\end{lemma}
\begin{proof}
   For all~$x_0 > 1/2$, we have
$$
n^{-1} \int_{n (x_0-1/2)}^{n (x_0 +1/2)} y^2 v(y) dy \geq n^2 (x_0-1/2)^2 v(n (x_0 +1/2)),
$$
thanks to the fact that~$v$ is non-increasing on~$(0,+\infty)$. Hence,
we have
\begin{multline*}
  n^2 (x_0 +1/2)^2 v(n (x_0 +1/2)) \leq \frac{n^2 (x_0 +1/2)^2}{n^2
    (x_0-1/2)^2} \; n^{-1} \int_{n (x_0-1/2)}^{n (x_0 +1/2)} y^2 v(y)
  dy.
\end{multline*}
Taking the upper semi-limit, we get for all~$x_0 > 1/2$
\begin{equation}\label{eq:tech:decay1}
  \limsup_{n \rightarrow {+\infty}} n^2 (x_0 +1/2)^2 v(n (x_0 +1/2)) 
\leq \ell \left(\frac{x_0 +1/2}{x_0-1/2} \right)^2,
\end{equation}
thanks to~\eqref{eq:tech:decay}. In the same way, one can show that
for all~$x_0> 1/2$,
\begin{equation}\label{eq:tech:decay2}
  \ell \left(\frac{x_0 -1/2}{x_0+1/2} \right)^2 
\leq \liminf_{n \rightarrow {+\infty}} n^2 (x_0-1/2)^2 v(n (x_0 -1/2)).
\end{equation}
Moreover, for fixed~$x_0>1/2$ and all~$y \geq x_0+1/2$, there exists
an unique integer~$n_y$ such that
$$
n_y (x_0 +1/2) \leq y < (n_y+1) (x_0 +1/2).
$$ 
Using again that~$v$ is non-increasing on~$[0,+\infty)$, we infer that
\begin{eqnarray*}
  y^2 v(y) & \leq & (n_y+1)^2 (x_0 +1/2)^2 \, v(n_y (x_0 +1/2)),\\
  & = & \frac{(n_y+1)^2 (x_0 +1/2)^2}{n_y^2 (x_0 +1/2)^2}\; n_y^2 (x_0 +1/2)^2 \, v(n_y (x_0 +1/2)).
\end{eqnarray*}
Notice that~$n_y \rightarrow {+\infty}$ as~$y \rightarrow
+\infty$. Therefore, passing to the upper semi-limit as~$y \rightarrow
+\infty$ in the inequality above, one can show that for all~$x_0>1/2$
$$
\limsup_{y \rightarrow +\infty} y^2 v(y) \leq \ell \left(\frac{x_0
    +1/2}{x_0-1/2} \right)^2,
$$
thanks to \eqref{eq:tech:decay1}. In the same way, we deduce
from~\eqref{eq:tech:decay2} that for all~$x_0>1/2$
$$
\ell \left(\frac{x_0 -1/2}{x_0+1/2} \right)^2 \leq \liminf_{y
  \rightarrow +\infty} y^2 v(y).
$$
Letting finally~$x_0 \rightarrow +\infty$ in both inequalities above
implies that
$$
\ell \leq \liminf_{y \rightarrow +\infty} y^2 v(y) \leq \limsup_{y
  \rightarrow +\infty} y^2 v(y) \leq \ell.
$$
Since~$v$ is even, we have completed  the proof of the lemma.
\end{proof}
For all $r>0$, the operator $\Lambda^1$ is the sum of $\Lambda_r^{(0)}$ and $\Lambda_r^{(1)}$. 
 As far as $\Lambda_r^{(1)}$ is concerned, we have the following lemma.
\begin{lemma}\label{lem:esti:convexe:concave}
  Let~$u \in L^\infty(\R)$ be non-decreasing, odd and convex
  on~$(-\infty,0)$. Then, for the operator defined in \rf{LK1}, we
  have $\Lambda_r^{(1)} u \in L_{loc}^1(\R_\ast)$ together with the
  inequality
\begin{eqnarray}\label{esti:fl1:loc}
\int_{|x| >  R} |\Lambda_r^{(1)} u(x)|\,dx  \leq  \frac{4 G_1 r}{R-r} \|u\|_\infty 
\end{eqnarray}
 for all~$R>r>0$.
\end{lemma}
\begin{proof} The proof is divided into a sequence of steps.

\medskip

{\it Step 1: estimates of~$u_x$.}  The convex function~$u$
on~$(-\infty,0)$ is locally Lipschitz-continuous on~$(-\infty,0)$
and~{\it a fortiori} a.e. differentiable. Since $u(0)=0$, we have
for $x <0$
\begin{equation}\label{esti:lip:loc}
|u_x(x)| \leq \|u\|_\infty \, |x|^{-1};
\end{equation}
Remark that this estimate holds true for $x \in \R$ since~$u$ is odd.

\medskip

{\it Step 2: estimates of~$u_{xx}$.} By convexity of~$u$,~$u_{xx}$ is
a non-negative Radon measure on~$(-\infty,0)$ in the distribution
sense. Hence,~$u_x \in BV_{loc}((-\infty,0))$ satisfies $
\int_{(\widetilde{x},x]} u_{yy}(dy)= u_x(x)-u_x(\widetilde{x}), $ for
a.e.~$\widetilde{x}<x<0$. Using~\eqref{esti:lip:loc} and
letting~$\widetilde{x} \rightarrow -\infty$, we conclude that for
a.e.~$x<0$
\begin{equation}\label{lem:esti:3}
\int_{(-\infty,x]} u_{yy}(dy)= u_x(x),
\end{equation}
thanks to the sup-continuity of non-negative measures. Again
by~\eqref{esti:lip:loc} and oddity of~$u_{xx}$, this shows for a.e.~$x
\neq 0$
\begin{equation}\label{esti:bv}
\int_{|y| \geq |x|} |u_{yy}|(dy)\leq 2\|u\|_\infty |x|^{-1};
\end{equation}
notice that by the inf-continuity of non-negative measures, this
inequality holds for all~$x \neq 0$.

\medskip

{\it Step 3: estimate of~$\Lambda_r^{(1)} u$.} Let us prove
that~$\Lambda_r^{(1)} u$ is well-defined by formula \eqref{LK1} for
a.e.~$x \neq 0$. By the preceding steps, we know that~$u \in
L^{\infty}(\R) \cap W_{loc}^{1,\infty}(\R_\ast)$ and~$u_x \in
BV_{loc}(\R_\ast)$. By Taylor's formula (see Lemma~\ref{lem:taylor} in
Appendix~\ref{appendix}), we infer that for all~$R>r>0$
\begin{align*}
  I & \equiv \int_{|x| > R} \int_{|z| \leq r} \frac{|u(x+z)-u(x)-u_x(x) z|}{|z|^{2}} \;  dx \; dz\\
  & \leq \int_{|x| > R} \int_{|z| \leq r} |z|^{-2}
  \left|\int_{I_{x,z}} |x+z-y| u_{yy}(dy) \right| \; dx \; dz,
\end{align*}
where~$I_{x,z}\equiv (x,x+z)$ if~$z>0$ and~$I_{x,z}\equiv (x+z,x)$ in
the opposite case.  Therefore, we see that
\begin{align*}
  I
  & \leq   \int_{|x| > R} \int_{|z| \leq  r} |z|^{-1} \int_{I_{x,z}} |u_{yy}|(dy)  \;  dx \; dz\\
  & = \int_{\R_\ast} \int_{\R} |z|^{-1} \mathbf{1}_{\{|z| \leq r\}}
  \int_{|x|>R} \mathbf{1}_{I_{x,z}}(y) \; dx \; |u_{yy}|(dy) \; dz,
\end{align*}
by integrating first w.r.t~$x$; notice that all the integrands are
measurable by Fubini's theorem, since the Radon measure~$|u_{yy}|
(dy)$ is $\sigma$-finite on~$\R_\ast$.  For fixed~$(y,z) \in \R_\ast
\times \R$, we have
\begin{eqnarray*}
  \mathbf{1}_{\{|z| \leq r\}} \int_{|x|>R} \mathbf{1}_{I_{x,z}}(y)  \;  dx 
\leq  |z| \; \mathbf{1}_{\{|z| \leq r\}} \; \mathbf{1}_{\{|y| \geq R-r\}},
\end{eqnarray*}
because the measure of the set~$\{x: y \in I_{x,z}\}$ can be estimated
by $|z|$, and if~$|z| \leq r$, then~$\mathbf{1}_{I_{x,z}}(y)=0$ for
all~$|x|>R$ whenever~$|y| <R-r$. It follows that
$$
I \leq \int_{\R_\ast} \int_{\R} \mathbf{1}_{\{|z| \leq r\}} \;
\mathbf{1}_{\{|y| \geq R-r\}} \; |u_{yy}|(dy) \; dz = 2r \int_{|y|
  \geq R-r} |u_{yy}|(dy).
$$
Recalling the definition of~$I$ above and estimate~\eqref{esti:bv}, we
have shown that
\begin{equation}\label{fubini}
  \int_{|x| > R} \int_{|z| \leq r} \frac{|u(x+z)-u(x)-u_x(x) z|}{|z|^{2}} \;  dx \; dz 
  \leq 4 r \|u\|_\infty (R-r)^{-1}.
\end{equation}
Fubini's theorem then implies that~$\Lambda_r^{(1)} u (x)$ is
well-defined by~\eqref{LK1} for a.e.~$|x| > R>r$ by satisfying the
desired estimate~\eqref{esti:fl1:loc}.

\medskip

{\it Step 4: local integrability on~$\R_\ast$.}
Estimate~\eqref{esti:fl1:loc} implies that~$\Lambda_r^{(1)} u \in
L^1_{loc}(\R \setminus [-r,r])$. In fact,~$\Lambda_r^{(1)} u$ is
locally integrable on all~$\R_\ast$. Indeed, simple computations show
that for all~$r>\widetilde{r}>0$
\begin{equation}\label{lem:esti:7}
  \Lambda_r^{(1)} u+\Lambda_r^{(0)} u=\Lambda_{\widetilde{r}}^{(1)} u+\Lambda_{\widetilde{r}}^{(0)} u,
\end{equation}
since their difference evaluated at some~$x$ is equal to $
\int_{\widetilde{r}\leq |z| \leq r} \frac{-u_x(x) z}{|z|^2}, $ which
is null by oddity of the function~$z \rightarrow -u_x(x) z$. By Step~3, 
it follows that~$\Lambda_r^{(1)} u=\Lambda_{\widetilde{r}}^{(1)}
u+\Lambda_{\widetilde{r}}^{(0)} u-\Lambda_r^{(0)} u \in L^1_{loc}(\R
\setminus [-\widetilde{r},\widetilde{r}])$, which completes the proof.
\end{proof}
It is clear that $\Lambda_r^{(0)}$ maps $L^\infty(\R)$ into $L^\infty (\R)$ and if
 $\{u_n\}_{n\in\mathbb{N}}$ is uniformly
  essentially bounded and $u_n \to u$ in $L^1_{loc}(\R)$, then
  $\Lambda_r^{(0)} u_n \to\Lambda_r^{(0)} u$ in $L^1_{loc}(\R)$ as $n
  \to + \infty$.
\begin{rem}\label{fl:distribution} 
    Lemma~\ref{lem:esti:convexe:concave} implies that~$\Lambda^1 u \in
    L^1_{loc}(\R_\ast)$ whenever $u \in L^\infty(\R)$ is
    non-decreasing, odd and convex on~$(-\infty,0)$. This sum does not
    depend on~$r>0$ by~\eqref{lem:esti:7}. Moreover, one sees
    from~\eqref{fubini}, Fubini's theorem and~\eqref{eqn:LK}, that for
    all~$\varphi \in \mathcal{D}(\R_\ast)$, $ \int_\R \varphi
    \Lambda^1 u \; dx =\int_\R u \Lambda^1 \varphi \; dx.  $ This
    means that this sum corresponds to the distribution fractional
    Laplacian of~$u$ on~$\R_\ast$.
\end{rem}
We deduce from the previous lemma the following one
\begin{lemma}\label{correction-point3-1}
  Let~$u \in C_b(\R)$ be non-decreasing, odd and convex
  on~$(-\infty,0)$. Then, the function $\Lambda^1 u \in L^1_{loc}(\R_\ast)$
  satisfies for all~$x_0>1/2$,
    $$
    \lim_{n \rightarrow +\infty} n^{-1} \int_{n(x_0-1/2)}^{n(x_0+1/2)} |\Lambda^1 u(y)| dy=0.
    $$
\end{lemma}
\begin{proof}
  By Remark~\ref{fl:distribution}, one has~$\Lambda^1 u \in
  L^1_{loc}(\R_\ast)$. Let~$r>0$ be fixed. One has
\begin{align*}
& I_n \equiv n^{-1} \int_{n(x_0-1/2)}^{n(x_0+1/2)} |\Lambda^1 u(y)| dy \\ 
&\leq n^{-1} \int_{n(x_0-1/2)}^{n(x_0+1/2)} |\Lambda_r^{(1)} u(y)| dy
+n^{-1} \int_{n(x_0-1/2)}^{n(x_0+1/2)} |\Lambda_r^{(0)} u(y)| dy\\
& \leq \frac{4 G_1 r}{n^2(x_0-1/2)-nr} \, \|u\|_\infty
+\sup \left\{|\Lambda_r^{(0)}(y)|:n(x_0-1/2)<y<n(x_0+1/2) \right\} ,
\end{align*}
thanks to~\eqref{esti:fl1:loc}. Moreover, $\Lambda_r^{(0)} u$ is
continuous, hence the supremum above is achieved at some~$y_n \geq
n(x_0-1/2)$; hence, one has
\begin{equation*}
  I_n 
  \leq \frac{4 G_1 r}{n^2(x_0-1/2)-nr} \, \|u\|_\infty+G_1 \int_{|z|>r} \frac{|u(y_n+z)-u(y_n)|}{|z|^2} \, dz,
\end{equation*}
where~$\lim_{n \rightarrow +\infty} y_n=+\infty$. Since~$u$ is
non-decreasing and bounded, it has a limit at infinity; the dominated
convergence theorem then implies that the integral term above tends to
zero as~$n \rightarrow +\infty$. It follows that $ \lim_{n
  \rightarrow +\infty} \, I_n =0.  $
\end{proof}

\subsection{Proof of property~p5 from Theorem~\ref{thm:qualitative}}
We assume
again without loss of generality that~$u_-=-u_+<0$, thanks to the
transformation~\eqref{transformation}; hence,~$U_0 \in L^\infty(\R)$
is non-decreasing, odd and convex on~$(-\infty,0)$ and so is~$U(t)$
for all~$t>0$ by properties p2--p4 of
Theorem~\ref{thm:qualitative}. We proceed again in several steps.
\medskip

{\it Step 1: study of $\Lambda^1 U$.} 
Before deriving the equation satisfied by $U(1)$, we study 
$\Lambda^1 U$. 
\begin{lemma}\label{cor:convergence:fl}
Let~$\alpha=1$ and~$U$ be the self-similar solution from
Theorem~\ref{thm:a=1:self} with initial datum~$U_0$ in~\eqref{ini:1}
for some~$u_-=-u_+<0$. Then, for all $t \geq 0$, one has~$\Lambda^1
U(t) \in L^1_{loc}(\R_\ast)$. Moreover,~$ \Lambda^1 U(t)$ converges
toward~$\Lambda^1 U_0$ in~$L^1_{loc}(\R_\ast)$ as~$t \rightarrow 0$,
where for all~$x \neq 0$
$$
\Lambda^1 U_0(x)= \frac{u_+-u_-}{2 \pi^2} \, x^{-1}.
$$ 
\end{lemma}
\begin{proof} By properties~p2--p4 of
  Theorem~\ref{thm:qualitative},~$U(t) \in L^\infty(\R)$ is
  non-decreasing, odd and convex on~$(-\infty,0)$ for all~$t \geq
  0$. By Remark~\ref{fl:distribution},~$\Lambda^1 U(t)$ and~$\Lambda^1
  U_0$ belong to~$L^1_{loc}(\R_\ast)$. By taking~$0<r<|x|$, simple
  computations show that
  \begin{equation}\label{lem:esti:6}
    \Lambda_r^{(1)} U_0(x)=0 \quad \mbox{and} \quad \Lambda_r^{(0)} U_0(x)=\frac{u_+-u_-}{2 \pi^2} \, x^{-1},
  \end{equation}
  so that
$$
\Lambda^1 U_0(x)=\frac{u_+-u_-}{2 \pi^2} \, x^{-1};
$$ here, we have used the equalities $\Gamma(1)=1$
and~$\Gamma(1/2)=\sqrt{\pi}$ in order to get $G_1=(2 \pi^2)^{-1}$
in~\eqref{LK1}--\eqref{LK2}. Moreover, Theorem~\ref{thm:entropy}
implies that~$U(t) \rightarrow U_0$ as~$t \rightarrow 0$
in~$L^1_{loc}(\R)$ with~$\|U(t)\|_\infty \leq \|U_0\|_\infty$. We
remark that for fixed~$r>0$,~$\Lambda_r^{(0)} U(t) \rightarrow
\Lambda_r^{(0)} U_0$ in~$L^1_{loc}(\R)$ as~$t \rightarrow 0$. It
follows that for all~$\widetilde{R}>R>r$,
  \begin{align*}
    & \limsup_{t \rightarrow 0} \int_{R<|x|<\widetilde{R}} |\Lambda^{1}U(t)-\Lambda^{1} U_0| \; dx\\
    & \leq \limsup_{t \rightarrow 0} \int_{R<|x|<\widetilde{R}} |\Lambda_r^{(1)} U(t)-\Lambda_r^{(1)} U_0| \; dx,\\
    & = \limsup_{t \rightarrow 0} \int_{R<|x|<\widetilde{R}}
    |\Lambda_r^{(1)} U(t)| \; dx
    \quad \mbox{by~\eqref{lem:esti:6}},\\
    & \leq \limsup_{t \rightarrow 0} 4 G_1 r \|U(t)\|_\infty
    (R-r)^{-1}
    \quad \mbox{by~\eqref{esti:fl1:loc} in Lemma~\ref{lem:esti:convexe:concave}},\\
    & \leq 4 G_1 r \|U_0\|_\infty (R-r)^{-1}.
  \end{align*}
  The proof is completed by letting~$r \rightarrow 0$.
\end{proof}

{\it Step 2: equation satisfied by~$U(1)$.} By using~$\eta(r) = \pm r$
in Definition~\ref{def:entropy}, we  obtain (in a classical way) 
that entropy
solutions to~\eqref{e1.1} are distribution solutions,~{\it i.e.}
\begin{equation}\label{eq:distribution}
U_t+U U_x+\Lambda^1 U=0 \quad \mbox{in} \quad \mathcal{D}'(\R \times (0,+\infty)).
\end{equation}
By property~p1 of Theorem~\ref{thm:qualitative}, one has $U(1) \in
W^{1,\infty}(\R)$. By the self-similarity $U(x,t)=U\left(\frac{x}{t},1\right)$, one has
at least~$U_t,U_x \in L^\infty_{loc}(\R \times (0,{+\infty}))$ together
with the following equalities for a.e.~$t>0$ and~$x \in \R$
\begin{equation*}
U_t(x,t)  =  -x t^{-2} U_x\Big(\frac{x}t,1\Big),\qquad
U_x(x,t)  =  t^{-1} U_x\Big(\frac{x}t,1\Big).
\end{equation*}
By Lemma~\ref{cor:convergence:fl}, we have also~$\Lambda^1 U(1)
\in L^1_{loc}(\R_\ast)$. Using again the self-similarity, it is easy
to deduce from~\eqref{eqn:LK} that~$\Lambda^1 U \in L^1_{loc}(\R_\ast
\times (0,{+\infty}))$ with for a.e.~$t>0$ and~$x \in \R_\ast$,
\begin{eqnarray*}
\Lambda^1 U(x,t) =  t^{-1} \Lambda^1 U\Big(\frac{x}t,1\Big)
\end{eqnarray*}
(in fact,~$\Lambda^1 U \in L^\infty_{loc}(\R \times (0,{+\infty}))$
by~\eqref{eq:distribution} so that~$\Lambda^1 U(1) \in
L^\infty_{loc}(\R)$). Putting these formulas
into~\eqref{eq:distribution}, we get for a.e.~$t>0$ and~$x \in \R$,
$$
-x t^{-2} U_x\Big(\frac{x}t,1\Big)
+t^{-1} U\Big(\frac{x}t,1\Big) U_x\Big(\frac{x}t,1\Big)
+t^{-1} \Lambda^1 U\Big(\frac{x}t,1\Big)=0.
$$
Multiplying by~$t$ and changing the variable by~$y=t^{-1} x$, one
infers  that the profile~$\U(y)\equiv U(y,1)$ satisfies for a.e~$y \in
\R$
\begin{equation}\label{eqn:self}
(\U(y)-y) \U_y(y)+\Lambda^1 \U(y)=0.
\end{equation} 

{\it Step 3: reduction of the problem.} By properties~p1--p4, the
function $\U_y \in L^\infty(\R)$ is non-negative, even and
non-decreasing on~$(-\infty,0)$. Then, Lemma~\ref{lem:equivalence} 
shows that the proof of~p5 can be reduced to the proof of the
following property:
\begin{equation}\label{decay2}
  \forall x_0>1/2 \quad \lim_{n \rightarrow {+\infty}} n^{-1}\int_{n (x_0-1/2)}^{n (x_0+1/2)} y^{2}\U_{y}(y) dy 
= \frac{u_+-u_-}{2 \pi^2}.
\end{equation}
Moreover, equality \eqref{eqn:self} implies
that~$\U_{y}(y)=\frac{\Lambda^1 \U(y)}{y -\U(y)}$ (for
a.e.~$y>\|\U\|_\infty$) and Lemma~\ref{correction-point3-1} implies that
$$
\lim_{n \rightarrow {+\infty}} n^{-1}\int_{n (x_0-1/2)}^{n (x_0+1/2)} |\Lambda^1 \U(y)| \, dy=0;
$$
hence, since~$\frac{y^2}{y-\mathcal{U}(y)}=y+\mathcal{O}(1)$ as~$|y| \rightarrow +\infty$, one 
deduces that~\eqref{decay2} is equivalent
to the following property:
\begin{equation}\label{decay3}
  \forall x_0>1/2 \quad 
  \lim_{n \rightarrow {+\infty}} n^{-1}\int_{n (x_0-1/2)}^{n (x_0+1/2)} y \Lambda^1 \U(y) \, dy=\frac{u_+-u_-}{2 \pi^2}.
\end{equation}

{\it Conclusion: proof of~\eqref{decay3}.} Let us change the variable by~$y=n x$. Easy computations show that
\begin{eqnarray*}
  n^{-1} \int_{n (x_0-1/2)}^{n (x_0+1/2)} y \Lambda^1 \U(y) \, dy & = &
  n^{-1} \int_{x_0-1/2}^{x_0+1/2} n x \Lambda^1 U \left(\frac{x}{n^{-1}},1\right) \, n dx,\nonumber \\
  & = & \int_{x_0-1/2}^{x_0+1/2} x \, \Lambda^1 U(x,n^{-1}) \, dx. \nonumber
\end{eqnarray*}  
Since lemma~\ref{cor:convergence:fl} implies that  $\{\Lambda^1 U(x,n^{-1})\}_{n \in \mathbb{N}}$ converges
 in~$L^1((x_0-1/2,x_0+1/2))$ toward~$\frac{u_+-u_-}{2 \pi^2}$ as~$n
\rightarrow {+\infty}$, the proofs of~\eqref{decay3} and thus of
property~p5 are complete.

\subsection{Duhamel's representation of the self-similar profile}

It remains to prove Theorem~\ref{thm:comp:cauchy}, for which we need
the following result.
\begin{prop}\label{prop:duhamel}
  Let~$\alpha=1$ and let~$U$ be the self-similar solution of
  Theorem~\ref{thm:a=1:self} with $u_\pm=\pm1/2$. Then, for all~$x \in
  \R$, we have
\begin{multline}\label{duh:self}
U(x,1)=-1/2+H_1(x,1)\\-\int_0^{1/2} \partial_x p_1(1-\tau) \ast  \frac{U^2(\cdot/\tau,1)}{2} \, (x)\; d\tau\\
-\int_{1/2}^1 \tau^{-1} \, p_1(1-\tau) \ast (U(\cdot/\tau,1)U_x(\cdot/\tau,1)) \, (x) \, d\tau
\end{multline}
(where~$H_1(x,1)=\int_{-\infty}^x p_1(y,1) dy$).
\end{prop}
\begin{proof} 
  The proof proceeds in several steps.

  {\it Step 1: Duhamel's representation of the approximate solution.}
  Notice that formula~\eqref{duh:self} makes sense. Indeed, by the
  homogeneity property~\eqref{homogeneity}, we have for all~$t>0$
  \begin{equation}\label{est:grad:cauchy}
    \|\partial_x p_1(t)\|_1 = C_0 t^{-1},
  \end{equation}
  where $C_0\equiv \|\partial_x P_1(1)\|_1$ is finite
  by~\eqref{EP}. Hence, the integral~$\int^{1/2}_0 \dots d\tau$
  in~\eqref{duh:self} is well-defined since the integration
  variable~$\tau$ is far from the singularity at $\tau=1$. In the same
  way, since $U (1) \in W^{1,\infty} (\R)$, the
  integral~$\int_{1/2}^{1} \dots d\tau$ is also well-defined.

  Let now~$u^\varepsilon$ be the solution to the
  regularized equation~\eqref{ue1}, with initial datum~$U_0$
  in~\eqref{ini:1}. The goal is to pass to the limit in
  formula~\eqref{duh:eps} at time~$t=1$, namely
  \begin{multline}\label{duh:eps:bis}
    u^\varepsilon(x,1)=S_1^\varepsilon(1) U_0(x)\\
    -\int_0^{1/2} p_2(\varepsilon(1-\tau)) \ast \partial_x p_1(1-\tau)
    \ast
    \frac{(u^\varepsilon(\tau))^2}{2} \, (x)\; d\tau\\
    -\int_{1/2}^1 p_2(\varepsilon(1-\tau)) \ast p_1(1-\tau) \ast
    (u^\varepsilon(\tau)u^\varepsilon_x(\tau)) \, (x) \, d\tau,
  \end{multline}
  for all~$x \in \R$. 

  {\it Step 2: pointwise limits and bounds of the integrands.}
  We first remark that
$$
\lim_{x\to \pm \infty} u^\eps (x,t) = u^\pm \, .
$$
Indeed, we know that $u^\eps$ is non-decreasing and it can be
shown for instance that $u^\eps -U_0 \in L^1 (\R)$. This fact
can be proved by splitting methods for instance.

Hence, thanks to Dini theorem for cumulative distribution functions,
we know that for fixed~$t>0$, $\lim_{\varepsilon \rightarrow 0}
u^\varepsilon(t)$ converges toward $U(t)$ uniformly on~$\R$.

  Let us next recall that~$\partial_x p_1(t) \in L^1(\R)$, so that for
  fixed~$\tau \in (0,1)$
$$
\lim_{\varepsilon \rightarrow 0} \partial_x p_1(1-\tau) \ast
\frac{(u^\varepsilon(\tau))^2}{2}=\partial_x p_1(1-\tau) \ast
\frac{(U(\tau))^2}{2} \quad \mbox{uniformly on~$\R$} .
$$
 It follows from classical approximate unit
properties of the heat kernel~$p_2(x,t)$ that for all~$\tau \in
(0,1)$,
\begin{multline}\label{cauchy2}
  \lim_{\varepsilon \rightarrow 0} p_2(\varepsilon(1-\tau))
  \ast \partial_x p_1(1-\tau) \ast \frac{(u^\varepsilon(\tau))^2}{2}
  = \partial_x p_1(1-\tau) \ast \frac{(U(\tau))^2}{2}
\end{multline}
uniformly on~$\R$. In particular, for all~$\tau \in (0,1)$, we have
also
\begin{multline}
\label{cauchy3}
\lim_{\varepsilon \rightarrow 0} p_2(\varepsilon(1-\tau)) \ast
p_1(1-\tau) \ast (u^\varepsilon(\tau) u_x^\varepsilon(\tau)) =
p_1(1-\tau) \ast (U(\tau) U_x(\tau))
\end{multline}
uniformly on~$\R$, since
\begin{multline*}
  p_2(\varepsilon(1-\tau)) \ast \partial_x p_1(1-\tau) \ast  \frac{(u^\varepsilon(\tau))^2}{2}\\
  =p_2(\varepsilon(1-\tau)) \ast p_1(1-\tau) \ast (u^\varepsilon(\tau)
  u_x^\varepsilon(\tau))
\end{multline*}
and $\partial_x p_1(1-\tau) \ast \frac{U^2(\tau)}{2}=p_1(1-\tau) \ast
(U(\tau) U_x(\tau)) $.

Moreover, by~\eqref{est:inf}, \eqref{ux dec} with $p = + \infty$ and
\eqref{est:grad:cauchy}, one can see that the integrands
of~\eqref{duh:eps:bis} are pointwise bounded by
\begin{equation}
  \Big\| p_2(\varepsilon (1-\tau)) \ast \partial_x p_1(1-\tau) \ast
  \frac{(u^\varepsilon(\tau))^2}{2}\Big\|_\infty \leq C_0(1-\tau)^{-1}
  \frac{\|u_0\|^2_\infty}{2},\label{cauchy4}
\end{equation}
and
\begin{equation}
  \Big\| p_2(\varepsilon (1-\tau)) \ast p_1(1-\tau) \ast
  (u^\varepsilon(\tau) u_x^\varepsilon(\tau)) \Big\|_\infty \leq
  \tau^{-1} \|u_0\|_\infty.
\label{cauchy5}
\end{equation}

{\it Step 3: passage  to the limit.} Recall that~
$$
\lim_{\varepsilon \rightarrow 0} S_1^\varepsilon(1) U_0 = S_1(1)
U_0=p_1(1) \ast U_0
$$ 
in~$L^p(\R)$ for all~$p \in [1,{+\infty}]$. Let us recall
that~$U_0(x)=\pm1/2$ for~$\pm x \geq 0$ and $\int_\R p_1(y,1) dy=1$,
so that for all~$x \in \R$
$$
1/2+p_1(1) \ast U_0(x)= p_1(1) \ast (U_0+1/2) (x)=\int_{-\infty}^x
p_1(y,1)dy=H_1(x,1).
$$
We have proved in particular that~$\lim_{\varepsilon \rightarrow 0}
S_1^\varepsilon(1) U_0 =-1/2+H_1(1)$ pointwise on~$\R$.

In order to pass to the limit in the integral terms
of~\eqref{duh:eps:bis}, we use the Lebesgue dominated convergence
theorem. We deduce from~\eqref{cauchy2} and~\eqref{cauchy4} that for
all~$x \in \R$, the first integral term converges toward
$$
\int_0^{1/2} \partial_x p_1(1-\tau) \ast  \frac{(U(\tau))^2}{2} \, (x)\; d\tau
$$
as~$\varepsilon \rightarrow 0$. In the same way, we deduce
from~\eqref{cauchy3} and~\eqref{cauchy5} that the last integral term
converges toward
$$
\int_{1/2}^1  p_1(1-\tau) \ast (U(\tau)U_x(\tau)) \, (x) \, d\tau.
$$

The limit as~$\varepsilon \rightarrow 0$ in~\eqref{duh:eps:bis} then
implies that for all~$x \in \R$,
\begin{multline*}
  U(x,1)=-1/2+H_1(x,1)
  -\int_0^{1/2} \partial_x p_1(1-\tau) \ast  \frac{U^2(\tau)}{2} \, (x)\; d\tau\\
  -\int_{1/2}^1 p_1(1-\tau) \ast (U(\tau)U_x(\tau)) \, (x) \, d\tau.
\end{multline*} 
This completes the proof of~\eqref{duh:self}, thanks to the
self-similarity of~$U$.
\end{proof}

\begin{proof}[Proof of Theorem~\ref{thm:comp:cauchy}] 
We have to prove
that for all~$r>0$
\begin{equation}\label{eq:comp:cauchy}
  \mathbb{P}(|X-\overline{c}|<r) <\mathbb{P}(|Y-0|<r).
\end{equation}
Let us verify that~$\overline{c}$ and~$0$ are the medians of~$X$
and~$Y$, respectively. First, a simple computation allows to see
that~$p_1(x,1)$, defined by Fourier transform
by~$\widehat{p_1}(\xi,1)=e^{-|\xi|}$, also satisfies
formula~\eqref{cauchy}. This density of probability is even and the
median of~$Y$ is null. Second, by property~p3 of
Theorem~\ref{thm:qualitative},~$U_x(1)$ is symmetric w.r.t. to the
axis~$\left\{x=\overline{c}\right\}$ and the median of~$X$
is~$\overline{c}=\frac{u_-+u_+}{2}$.

In particular, the centered random variable~$X-\overline{c}$ admits a
density being the even function
$$
f_{X-\overline{c}}(x)=U_x(x+\overline{c},1).
$$
It becomes clear that~\eqref{eq:comp:cauchy} is equivalent to the
following property
\begin{equation}\label{comp:tech:1}
  \forall x>0 \quad F_{X-\overline{c}}(x)<F_Y(x), 
\end{equation}
where~$F_{X-\overline{c}}$ and~$F_Y$ are the cumulative distribution
functions of~$X-\overline{c}$ and~$Y$, respectively.

Let us compute these functions. First, we have seen above
that~$f_{X-\overline{c}}(x)=V_x(x,1)$, where~$V$ is defined by the
transformation~\eqref{transformation}. Let us recall that~$V$ is the
self-similar solution to~\eqref{e1.1} with initial
datum~$V(x,0)=\pm1/2$ for~$\pm x>0$. Hence,~$F_{X-\overline{c}}$ is
equal to~$V(\cdot,1)$ up to an additive constant, which has to
be~$1/2$ by property~p2 of Theorem~\ref{thm:qualitative}; that is to
say, we have~$F_{X-\overline{c}}(x)=1/2+V(x,1)$ for all~$x \in
\R$. Second, we defined $H_1$ in Proposition~\ref{prop:duhamel} such
that~$F_Y(x)=H_1(x,1)$. By this proposition, we have for all~$x
\in \R$,
\begin{equation*}
F_{X-\overline{c}}(x)=F_Y(x)-g(x),
\end{equation*}
where~$g(x)$ is defined by
\begin{multline}
  \label{def:difference:cauchy}
  g(x)\equiv \int_0^{1/2} \partial_x p_1(1-\tau) \ast
  \frac{V^2(\cdot/\tau,1)}{2} \, (x)\; d\tau\\
  +\int_{1/2}^1 \tau^{-1} \, p_1(1-\tau) \ast
  (V(\cdot/\tau,1)V_x(\cdot/\tau,1)) \, (x) \, d\tau.
\end{multline}
One concludes that the proof of~\eqref{comp:tech:1}, and thus
of~\eqref{eq:comp:cauchy}, is equivalent to the proof of the
positivity of~$g(x)$ for positive~$x$. But, by definition of~$g$, it
suffices to prove that for each~$\tau \in (0,1)$ and~$x>0$,
\begin{equation}\label{eq:sign:g}
p_1(1-\tau) \ast (V(\cdot/\tau,1)V_x(\cdot/\tau,1)) \, (x)>0.
\end{equation} 
Indeed, the second integral term in~\eqref{def:difference:cauchy}
would be positive, and the first integral term also, since for
fixed~$\tau$,
$$
\partial_x p_1(1-\tau) \ast \frac{V^2(\cdot/\tau,1)}{2} \,
(x)=\tau^{-1} \, p_1(1-\tau) \ast (V(\cdot/\tau,1)V_x(\cdot/\tau,1))
\, (x).
$$

Let us end by proving inequality \eqref{eq:sign:g}, thus concluding
Theorem~\ref{thm:comp:cauchy}. It is clear
that the function $V(\cdot/\tau,1)V_x(\cdot/\tau,1)$ is odd, since~$V(1)$ is
odd. Moreover, we already know that~$V_x(1)$ is non-negative, even and
non-increasing on~$(0,+\infty)$, since~$V(1)$ is non-decreasing, odd
and concave on~$[0,+\infty)$. By property~p5, we conclude
that~$V_x(1)$ is positive a.e. on~$(0,+\infty)$, and thus on~$\R$ as
even function. In particular,~$V(1)$ is increasing and for
all~$x>0$,~$V(x,1)>V(0,1)=0$.

To summarize,~$V(\cdot/\tau,1)V_x(\cdot/\tau,1)$ is odd and positive
on~$(0,+\infty)$. Moreover, it is clear that~$p_1(1-\tau)$ is
positive, even and decreasing on~$(0,+\infty)$, see~\eqref{cauchy}. A
simple computation then implies that the convolution product
in~\eqref{eq:sign:g} is effectively positive for positive~$x$. The
proof of Theorem~\ref{thm:comp:cauchy} is complete.  
\end{proof}


\appendix

\section{A key estimate}\label{appendix-KMX}

Here is an estimate from the lines of~\cite[Proof of Lemma 3.1]{kmx}. 
\begin{lemma}[inspired from~\cite{kmx}]\label{KMX}
Let~$\alpha \in (0,2]$ and let us consider a function~$v$ such that for all~$a>0$,
$v \in C^\infty_b(\R \times (a,+\infty)) \cap L^\infty(0,+\infty;L^1(\R))$. 
Assume that for all~$p \in [2,+\infty)$ and~$t>0$,
\begin{equation}\label{assumptionii}
\int_{\R} v_t |v|^{p-2} v \, dx+\int_\R |v|^{p-2} v \Lambda^\alpha v dx \leq 0.
\end{equation}
Then there is a constant~$C=C(\alpha)>0$ such that for all~$p \in [1,+\infty]$ and all~$t>0$
\begin{equation}\label{kmx esti}
\|v(t)\|_p \leq Ct^{-\frac{1}{\alpha} \left(1-\frac{1}{p}\right)} \|v\|_{L^\infty(0,+\infty;L^1)}.
\end{equation}
\end{lemma}

The proof is based on the so-called Nash and Strook-Varopoulos inequalities.
\begin{lemma}[Nash inequality]\label{Nash}
Let~$\alpha>0$. There exists a constant~$C_N>0$ such that for all~$w \in L^1(\R)$ satisfying
~$\Lambda^{\alpha/2} w \in L^2(\R)$, one has
$$
\|w\|_2^{2(1+\alpha)} \leq C_N\|\Lambda^{\alpha/2} w\|_2^2 \|w\|_1^{2\alpha}.
$$
\end{lemma}

\begin{lemma}[Strook-Varopoulos inequality]\label{SV}
Let~$\alpha \in (0,2]$. For all~$p \in [2,+\infty)$ and~$w \in L^{p-1}(\R)$ satisfying
~$\Lambda^{\alpha} w \in L^\infty(\R)$, one has
$$
\int_\R |w|^{p-2} w \Lambda^\alpha w \, dx \geq \frac{4(p-1)}{p^2}  
\int_\R \left(\Lambda^{\alpha/2} |w|^{p/2} \right)^2 \, dx.
$$
\end{lemma}

\begin{rem}
\begin{enumerate}
\item In the case~$\alpha=2$, simple computations show that one has an equality in place of an inequality.
\item As
suggested by the proof below,
the second lemma is valid for all~$p \in [1,+\infty)$ with~$w,\Lambda^\alpha w \in L^p(\R)$, as well
as in the muldimensional case and for more
general operator~$\Lambda^\alpha$ satisfies the postive maximum principle
(see~\cite{h}).
\end{enumerate}
\end{rem}

Proofs and references for these results can be found in~\cite{kmx,k}. Let us give them for the sake of
completeness.

\begin{proof}[Proof of Lemma~\ref{Nash}]
Let us first prove the result for~$\varphi \in \mathcal{D}(\R)$. By Plancherel equality, one has
$$
\|\varphi\|_2^2=
\|\hat{\varphi}\|_2^2  \leq \int_{|\xi| <r} |\hat{\varphi}(\xi)|^ d\xi+r^{-\alpha} 
\int_{|\xi| \geq r} |\xi|^\alpha |\hat{\varphi}(\xi)|^2 d\xi, 
$$
for all~$r>0$. Then, one gets
$$
\|\varphi\|_2^2 \leq 2r \|\hat{\varphi}\|_\infty^2+r^{-\alpha} \|\Lambda^{\alpha/2} \varphi\|_2^2 
\leq 2r \|\varphi\|_1^2+r^{-\alpha} \|\Lambda^{\alpha/2} \varphi\|_2^2.
$$
Now an optimization w.r.t.~$r>0$ gives the result for~$\varphi$ smooth. The result for~$w$ as in the lemma is 
deduced by approximation. 
\end{proof}

\begin{proof}[Proof of Lemma~\ref{SV}]
Let us proceed in several steps.

\medskip

{\itshape Step 1: a first inequality.} Let us prove that for all~$\beta,\gamma>0$ such that
~$\beta+\gamma=2$, one has for all non-negative reals~$a,b$
\begin{equation}\label{ineq-k-1}
(a^\beta-b^\beta)(a^\gamma-b^\gamma) \geq \beta \gamma (a-b)^2.
\end{equation}
Let us assume without loss of generality that~$a > b>0$ and~$\beta \leq \gamma$. 
Developping each members of~\eqref{ineq-k-1}, one sees that this equation 
is equivalent to
\begin{multline*}
(1-\beta \gamma) \left(a^2+b^2\right) \\
\stackrel{?}{\geq} a^{\beta}b^{\gamma}+a^{\gamma}b^{\beta}-2 \beta \gamma ab=(ab)^\beta 
\left(a^{2(1-\beta)}+b^{2(1-\beta)}-2\beta \gamma (ab)^{1-\beta} \right).  
\end{multline*}
Since one has~$1-\beta \gamma=(1-\beta)^2$ and
$$
a^{2(1-\beta)}+b^{2(1-\beta)}-2\beta \gamma (ab)^{1-\beta}
=\left(a^{1-\beta}-b^{1-\beta} \right)^2+2(1-\beta \gamma)(ab)^{1-\beta},
$$
one deduces that~\eqref{ineq-k-1} is equivalent to
$$
(1-\beta)^2 \left(a^2+b^2-2ab\right)=(1-\beta)^2(a-b)^2 \stackrel{?}{\geq} 
(ab)^\beta \left(a^{1-\beta}-b^{1-\beta} \right)^2;
$$
that is to say, one has to prove that for all~$\beta \in (0,1]$ and~$a>b > 0$
$$
(1-\beta)(a-b)\stackrel{?}{\geq} 
(ab)^{\beta/2} \left(a^{1-\beta}-b^{1-\beta}\right).
$$
Dividing by~$b>0$ and denoting~$x$ the variable~$\frac{a}{b}$, one has to prove that for 
all~$\beta \in (0,1]$ and~$x>1$
$$
g(x) \equiv (1-\beta) (x-1)-x^{1-\beta/2}+x^{\beta/2} \stackrel{?}{\geq} 0.
$$
Since~$g$ is continuous w.r.t.~$x \in [1,+\infty)$ with~$g(1)=0$,
it suffices to prove that~$g'(x) \geq 0$ for all~$x >1$. One has
$$
g'(x)=1-\beta-\left(1-\beta/2\right)x^{-\beta/2}+\frac{\beta}{2} \, x^{-1+\beta/2}.
$$
Again~$g'$ is continuous with~$g(1)=0$, so that
the proof of~\eqref{ineq-k-1} reduces finally to the proof of the non-negativity of~$g''(x)$ for all~$x>1$. One has
$$
g''(x)=\frac{\beta}{2} (1-\beta/2) x^{-1-\beta/2}+\frac{\beta}{2} \, (-1+\beta/2) x^{-2+\beta/2},
$$
so that~$g''(x)\geq 0$ is equivalent to
$
x^{1-\beta} \stackrel{?}{\geq} 1,
$
which is true for~$\beta \in (0,1]$ and~$x>1$. The proof of~\eqref{ineq-k-1} is complete.

\medskip

{\itshape Conclusion.} Take~$\psi \in C_c(\R)$
and assume~$\psi \geq 0$. For all~$r>0$ and~$\beta,\gamma>0$, one has
\begin{align*}
& \int_{\R} \psi^\gamma \Lambda_r^{(0)} \psi^{\beta} \, dx\\ 
& = G_\alpha \int \int_{|x-y|>r}  \frac{(\psi^\beta(x)-\psi^\beta(y)) \psi^\gamma(x)}{|x-y|^{1+\alpha}} \, dxdy,\\
& = G_\alpha \int \int_{|x-y|>r}  \frac{(\psi^\beta(y)-\psi^\beta(x)) \psi^\gamma(y)}{|x-y|^{1+\alpha}} \, dxdy
\end{align*}
by changing the variable~$(x,y) \rightarrow (y,x)$ and using the fact that the measure
~$\frac{dxdy}{|x-y|}$ is symmetric. 
It follows that
\begin{equation*}
\int_{\R} \psi^\gamma \Lambda_r^{(0)} \psi^{\beta} \, dx \\
=
\frac{G_\alpha}{2} \int \int_{|x-y|>r} 
\frac{(\psi^\beta(y)-\psi^\beta(x)) (\psi^\gamma(y)-\psi^\gamma(x))}{|x-y|^{1+\alpha}}
\, dxdy.
\end{equation*}
On using Step 1, one deduces that for all~$\psi \in C_c(\R)$,
~$\psi \geq 0$, all~$\beta,\gamma>0$,~$\beta+\gamma=2$ and all~$t>0$, one has
\begin{equation}\label{ineq-k-2}
\int_{\R} \psi^\gamma \Lambda_r^{(0)} \psi^{\beta} \, dx
\geq \beta \gamma \int_{\R} \psi \Lambda_r^{(0)} \psi \, dx.
\end{equation}
Take now~$\varphi \in \mathcal{D}(\R)$,~$\varphi \geq 0$ and
~$p > 1$. Let us choose~$\psi=\varphi^{p/2}$,~$\beta=2/p$ and~$\gamma=2-\beta=2 \left(1-\frac{1}{p}\right)$. 
Equation~\eqref{ineq-k-2} gives:
\begin{equation*}
\int_{\R} \varphi^{p-1}  \Lambda_r^{(0)} \varphi \, dx
\geq \frac{4(p-1)}{p^2} \int_{\R} \varphi^{p/2} \Lambda_r^{(0)} \varphi^{p/2} \, dx.
\end{equation*}

Hence, for~$\varphi \in \mathcal{D}(\R)$ not necessarily non-negative, Kato inequality (with
~$\eta(\cdot) \equiv |\cdot|$ convex) implies 
\begin{align*}
& \int_\R |\varphi|^{p-2} \varphi \Lambda_r^{(0)} \varphi \, dx\\
& \geq \int_\R |\varphi|^{p-1}  \Lambda_r^{(0)} |\varphi| \, dx,\\
& \geq \frac{4(p-1)}{p^2} \int_{\R} |\varphi|^{p/2} \Lambda_r^{(0)} |\varphi|^{p/2} \, dx,\\
& = \frac{4(p-1)}{p^2}
\frac{G_\alpha}{2} \int \int_{|x-y|>r} 
\frac{(|\varphi|^{p/2}(y)-|\varphi|^{p/2}(x))^2}{|x-y|^{1+\alpha}}
\, dxdy.
\end{align*}
Passing to the limit as~$r \rightarrow 0$, one concludes that
\begin{align*}
& \int_{\R} |\varphi|^{p-2} \varphi  \Lambda^\alpha \varphi \, dx\\
& \geq \frac{G_\alpha}{2} \int_\R \int_\R
\frac{(|\varphi|^{p/2}(y)-|\varphi|^{p/2}(x))^2}{|x-y|^{1+\alpha}}
\, dxdy,\\
& =\frac{4(p-1)}{p^2} \int_{\R} \left(\Lambda^{\alpha/2}  |\varphi|^{p/2} \right)^2 \, dx.
\end{align*}
This proves the result for~$\varphi \in \mathcal{D}(\R)$ non-negative.
The proof for~$w$ as in the lemma is complete by approximation.
\end{proof}

Before proving Lemma~\ref{KMX}, one needs to establish a relationship between 
the differential inequality~\eqref{assumptionii}
and the~$L^p$-norm in space of~$v$:
\begin{lemma}\label{lemme-derivee-norme}
Let~$v$ such that~$v \in C^\infty_b(\R \times (a,+\infty)) \cap L^\infty((0,+\infty);L^1(\R))$ for all~$a>0$. 
Then for all~$p \in [2,+\infty)$, the function~$t>0 \rightarrow \|v(t)\|_p^p$
is locally Lipschitz-continuous with for a.e.~$t>0$
\begin{equation}\label{formule-derivee-norme}
\frac{1}{p} \frac{d}{dt} \|v(t)\|_p^p=\int_{\R} v_t(x,t) |v(x,t)|^{p-2} v(x,t) \, dx.
\end{equation}
\end{lemma}

\begin{proof}
Let~$\{\varphi_n\}_{n \in \mathbb{N}} \in \mathcal{D}(\R \times (0,+\infty))$ be a sequence such that
$$
\begin{cases}
\lim_{n} \varphi_n = v \mbox{ in $C^k(K)$ for all compact~$K \subset \R \times (0,+\infty)$ and~$k \in \mathbb{N}$},\\
\mbox{$\{\varphi_n\}_{n \in \mathbb{N}}$ is bounded in~$C^k(\R \times (a,+\infty))$ for all~$a>0$ 
and~$k \in \mathbb{N}$},\\
|\varphi_n| \leq |v| \mbox{ for all~$n \in \mathbb{N}$}.
\end{cases}
$$
(such a sequence is easily constructed by taking~$\varphi_n \equiv v \theta_n$,
with~$0 \leq \theta_n \leq 1$,~$\theta_n \rightarrow 1$ in~$C^k(K)$ and~$\{\theta_n\}_{n \in \mathbb{N}}$
bounded in~$C^k(\R \times (a,+\infty))$).
One has for all~$p \in [2,+\infty)$ and~$t,s>0$,
$$
\frac{\|\varphi_n(t)\|_p^p-\|\varphi_n(s)\|_p^p}{p}=
 \int_{\R} \int_{s}^t |\varphi_n|^{p-2} \varphi_n \partial_{\tau} {\varphi_n}  \, dx d\tau .
$$
By the dominated convergence theorem, one gets
\begin{align*}
& \lim_{n \rightarrow +\infty} \frac{\|\varphi_n(t)\|_p^p-\|\varphi_n(s)\|_p^p}{p}\\
& = 
\lim_{n \rightarrow +\infty} 
 \int_{\R} \int_{s}^t  |\varphi_n|^{p-2} \varphi_n \partial_{\tau} {\varphi_n}  \, dx d\tau,\\
& = \int_{\R} \int_{s}^t v_{\tau} |v|^{p-2} v \, dx d\tau.
\end{align*}
But the dominated convergence theorem also allows to prove that
$$
\lim_{n \rightarrow +\infty} \frac{\|\varphi_n(t)\|_p^p-\|\varphi_n(s)\|_p^p}{p}
=\frac{\|v(t)\|_p^p - \|v(s)\|_p^p}{p}.
$$
By uniqueness of the limit, one deduces that 
$$
\frac{\|v(t)\|_p^p - \|v(s)\|_p^p}{p} \, = \int_{s}^t \left( \int_{\R} v_{\tau} |v|^{p-2} v \, dx \right) d\tau.
$$
Since~$\tau \rightarrow \int_{\R} v_{\tau}(x,\tau) |v(x,\tau)|^{p-2} v(x,\tau) \, dx$ is 
bounded outside all neighborhood of~$\tau=0$, the proof is complete.
\end{proof}

\begin{proof}[Proof of Lemma~\ref{KMX}]
The proof follows~\cite[Proof of Lemma 3.1]{kmx}.
One deduces from~\eqref{assumptionii} and Lemmata~\ref{SV} 
and~\ref{lemme-derivee-norme} that for all~$p \in [2,+\infty)$
and a.e.~$t>0$,
\begin{equation}\label{di1}
\frac{d}{dt} \|v(t)\|_p^p+4 \left(1-\frac{1}{p} \right)
\int_\R \left(\Lambda^{\alpha/2} |v|^{p/2} \right)^2 \, dx \leq 0.
\end{equation}
Let us now prove~\eqref{kmx esti} for~$p=2^n$ by induction on~$n \geq 1$. 
In the sequel,~$C_0$ denotes the constant~$\|v\|_{L^\infty(0,+\infty,L^1)}$.
For~$p=2$, one uses~\eqref{di1} and Lemma~\ref{Nash} to get:
$$
\frac{d}{dt} \|v(t)\|_2^2+2
C_N^{-1} C_0 ^{-2\alpha}\|v(t)\|_2^{2(1+\alpha)}  \leq 0,
$$
which leads to
$$
\|v(t)\|_2 \leq C_1 C_0 t^{-\frac{1}{2\alpha}} \quad \mbox{with} \quad C_1\equiv 
\left(\frac{C_N}{2\alpha} \right)^{\frac{1}{2 \alpha}}.
$$
Suppose now that for~$n \geq 2$ there is a constant~$C_n$ such that for all~$t>0$
$$
\|v(t)\|_{2^n} \leq C_n C_0 t^{-\frac{1}{\alpha} \left(1-2^{-n}\right)}.
$$
Then, for~$p=2^{n+1}$,~\eqref{di1} and Lemma~\ref{Nash} applied to~$w=v^{2^n}$ gives:
$$
\frac{d}{dt} \|v(t)\|_{2^{n+1}}^{2^{n+1}}+4 \left(1-2^{-n-1} \right)
C_N^{-1} \|v\|_{2^n}^{-2^{n+1}\alpha} \;  \|v(t)\|_{2^{n+1}}^{2^{n+1}(1+\alpha)}  \leq 0.
$$
By the inductive hypothesis, one gets
$$
\frac{d}{dt} \|v(t)\|_{2^{n+1}}^{2^{n+1}}+4 \left(1-2^{-n-1} \right)
C_N^{-1} \left(C_n C_0 \right)^{-2^{n+1}\alpha} \; t^{2^{n+1}-2}  \;  
\left(\|v(t)\|_{2^{n+1}}^{2^{n+1}} \right)^{(1+\alpha)}  \leq 0,
$$
which leads to
$$
\|v(t)\|_{2^{n+1}} \leq C_{n+1} C_0 t^{-\frac{(1-2^{-n-1})}{\alpha}} \quad \mbox{with} \quad C_{n+1}=C_n 
\left(\frac{C_N}{2\alpha}\right)^{\frac{2^{-n-1}}{\alpha}} \left(2^{n2^{-n-1}} \right)^{\frac{1}{\alpha}}.
$$
Now it rests to prove that~$\limsup_{n \rightarrow +\infty} C_n<+\infty$; indeed, the limit~$n \rightarrow +\infty$
in the inequality above will gives~\eqref{kmx esti} for~$p=+\infty$ and the proof of the lemma will be complete
by interpolation of the~$L^1$- and
~$L^\infty$-norms. 

One has
$$
\ln C_{n+1} - \ln C_n= \ln \left( \frac{C_{n+1}}{C_n} \right) =
\frac{2^{-n-1}}{\alpha} \ln \left(\frac{C_N}{2\alpha}\right) 
+\frac{n2^{-n-1}}{\alpha} \ln 2 \equiv u_n,
$$
where the serie~$\Sigma u_n$ is convergent. Summing up all these inequalities for~$n=1,\dots,N$, one gets 
for all~$N \geq 1$,
$
\ln C_{N+1}=\ln C_1+\Sigma_{n=1}^{N} u_n.
$
The limit~$N \rightarrow +\infty$ then gives:
$$
\lim_{n \rightarrow +\infty} \ln C_{n}=\ln C_1+\Sigma_{k=1}^{+\infty} u_k \in \R,
$$
so that~$\lim_{n \rightarrow +\infty} C_n$ exits in~$\R$.
\end{proof}

\section{Proof of Theorem~\ref{thm:convg}}\label{correction-point2-appendix}

Inequality from the following proposition is the starting point to
prove Theorem \ref{thm:convg}.
\begin{prop}\label{thm:fi} Let $u_0, \widetilde{u}_0 \in L^\infty(\R)$ and  $\varepsilon>0$. 
Let
 $\ue$ and $\widetilde{\ue}$   be
the solutions to \rf{uge1}--\rf{uge2} with the  initial data $u_0$ 
and $\widetilde{u}_0$, resp. Then 
\begin{equation}\label{est:fi}
\int_{-R}^R |\ue(x,t)-\widetilde{\ue}(x,t)|\,dx 
\leq \int_{-R-Lt}^{R+Lt} S_\alpha^\varepsilon(t) |u_0-\widetilde{u}_0|(x)\,dx 
\end{equation}
for all $t>0$ and $R>0$,
where 
\begin{equation}\label{LM}
L=\max_{z\in[-M,M]} |f'(z)|\quad\mbox{and} 
\quad M=\max \left\{ \|u_0\|_\infty,\|\widetilde{u_0}\|_\infty\right\}.
\end{equation}
\end{prop}
Even if this result does not appear in \cite{A07}, its proof is based
on ideas introduced in \cite[Thm~3.2]{A07}. This is the reason why we
only sketch the proof of Proposition~\ref{thm:fi}; the reader is
referred to \cite{A07} for more details.
\begin{proof}[Sketch of proof of Proposition \ref{thm:fi}]
The solution $u^\varepsilon$ of \rf{uge1}--\rf{uge2} satisfies 
\begin{multline}\label{eqn:entropy:reg}
  \int_{\R} \int_a^{+\infty} \Big(\eta(u^\varepsilon) \varphi_t
  +\phi(u^\varepsilon) \varphi_x\Big)\, dxdt\\
  +\int_{\R} \int_a^{+\infty} \Big( -\eta(u^\varepsilon)
  \Lambda_r^{(\alpha)} \varphi
  -\varphi \eta'(u^\varepsilon)  \; \Lambda_r^{(0)} u^\varepsilon \Big)\,dxdt \\
  -\varepsilon \int_{\R} \int_a^{+\infty} \left(\eta (u^\varepsilon)
  \right)_{x} \varphi_x\, dxdt +\int_{\R} \eta(u^\varepsilon(x,a))
  \varphi(x,a)\,dx \geq 0,
\end{multline} 
for all~$\varphi \in \mathcal{D}(\R \times [0,{+\infty}))$
non-negative,~$\eta \in C^2(\R)$ convex,~$\phi'=\eta' f'$
and~$a,r>0$. To show this inequality, it suffices to
mutliply~\eqref{uge1} by~$\eta'(u^\varepsilon) \varphi$, use the
Kato inequalities~\eqref{ineq:kato} and integrate by parts over the
domain~$\R \times [a,{+\infty})$.  Now, let us introduce the so-called
Kruzhkov entropy-flux pairs~$(\eta_k,\phi_k)$ defined for fixed~$k \in
\R$ and all~$u \in \R$ by
$$
\eta_k(u)\equiv|u-k| \quad \mbox{and} \quad
\phi_k(u)\equiv\mbox{sign}(u-k) \, \left(f(u)-f(k) \right),
$$
where ``sign'' denotes the sign function defined by
$$
\mbox{sign}(u)\equiv
\begin{cases}
1, & \mbox{~$u>0$},\\
-1, & \mbox{~$u<0$},\\
0, & \mbox{~$u=0$}.
\end{cases}
$$
Consider a sequence~$\{\eta_k^n\}_{n\in\mathbb{N}} \subset C^2(\R)$ of
convex functions converging toward~$\eta_k$ locally uniformly on~$\R$
and such that~$(\eta_k^n)' \rightarrow \mbox{sign}(\cdot-k)$ pointwise
on~$\R$ by being bounded by~$1$, as~$n \rightarrow {+\infty}$. The
associated fluxes~$\phi_k^n(u)\equiv\int_k^u \eta_k'(\tau) f'(\tau)
d\tau$ then converge toward~$\phi_k$ pointwise on~$\R$, as~$n
\rightarrow {+\infty}$, by being pointwise bounded by~$|\phi_k^n(u)| \leq
\mbox{sign}(u-k) \, \int_k^u |f'(\tau)| d\tau$. By the dominated
convergence theorem, the passage to the limit
in~\eqref{eqn:entropy:reg} with~$(\eta,\phi)=(\eta_k^n,\phi_k^n)$
gives
\begin{multline}\label{ineq:entropy:reg:2}
  \int_{\R} \int_a^{+\infty} \Big(|u^\varepsilon-k| \varphi_t 
  +\mbox{sign}(u^\varepsilon-k) \left(f(u^\varepsilon)-f(k) \right) \varphi_x\Big)\, dxdt\\
  + \int_{\R} \int_a^{+\infty} \Big( -|u^\varepsilon-k|
  \Lambda_r^{(\alpha)} \varphi
  -\varphi \, \mbox{sign}(u^\varepsilon-k)  \; \Lambda_r^{(0)} u^\varepsilon \Big)\,dxdt \\
  -\varepsilon \int_{\R} \int_a^{+\infty} \mbox{sign}(u^\varepsilon-k) \,
  u_x^\varepsilon \, \varphi_x \, dxdt+\int_{\R}
  |u^\varepsilon(x,a)-k| \varphi(x,a)\,dx \geq 0,
\end{multline} 
for all~$\varphi \in \mathcal{D}(\R \times [0,{+\infty}))$
non-negative,~$a,r>0$ and~$k \in \R$. In the same way, similar
inequalities hold true for~$\widetilde{u}^\varepsilon$.

On the basis of these inequalities, we claim that the well-known
doubling variable technique of Kruzhkov allows us to compare
$u_\varepsilon$ and~$\widetilde{u}_\varepsilon$. To do so, we have to
copy almost the same computations from \cite{A07}, since the beginning
of~\cite[Subsection 4.1]{A07} until~\cite[equation (4.11)]{A07}
with~$u=u^\varepsilon$ and~$v=\widetilde{u^\varepsilon}$. The only
difference comes from the term~$-\varepsilon \int_{\R} \int_a^{+\infty}
\mbox{sign}(u^\varepsilon-k) \, u_x^\varepsilon \, \varphi_x \, dxdt$
in~\eqref{ineq:entropy:reg:2} and the term~$-\varepsilon \int_{\R}
\int_a^{+\infty} \mbox{sign}(\widetilde{u^\varepsilon}-k) \,
\widetilde{u_x^\varepsilon} \, \varphi_x \, dxdt$ in the entropy
inequalities of~$\widetilde{u^\varepsilon}$. But, these new terms do
not present any particular difficulty, since~$u_\varepsilon$
and~$\widetilde{u}_\varepsilon$ are smooth. Arguing as in~\cite{A07},
one can show that for all~$\phi \in \mathcal{D}(\R \times [0,{+\infty}))$
non-negative and~$a>0$,
\begin{multline*}
  \int_{\R} \int_a^{+\infty} |u^\varepsilon-\widetilde{u^\varepsilon}|
  \left(\phi_t+L|\phi_x|-\Lambda^\alpha \phi \right) \, dxdt\\
  -\varepsilon \int_{\R} \int_a^{+\infty} \mbox{sign}
  (u^\varepsilon-\widetilde{u^\varepsilon}) \, (u^\varepsilon-\widetilde{u^\varepsilon})_x \, \phi_x \, dxdt\\
  +\int_{\R} |u^\varepsilon(x,a)-\widetilde{u^\varepsilon}(x,a)|
  \phi(x,a) \, dx \geq 0,
\end{multline*}
where~$L$ is defined
in~\eqref{LM}. Since~$|u^\varepsilon-\widetilde{u^\varepsilon}|$ is
Lipschitz-continuous on~$\R \times [a,{+\infty})$, its a.e. derivative is
equal to its distribution derivative with $ \mbox{sign}
(u^\varepsilon-\widetilde{u^\varepsilon}) \,
(u^\varepsilon-\widetilde{u^\varepsilon})_x=\left(|u^\varepsilon-\widetilde{u^\varepsilon}|\right)_x.
$ By integrating by parts, we deduce that
\begin{multline*}
  \int_{\R} \int_a^{+\infty} |u^\varepsilon-\widetilde{u^\varepsilon}| \left(\phi_t+L|\phi_x|-g [\phi] \right)\,dxdt\\
  +\int_{\R} |u^\varepsilon(x,a)-\widetilde{u^\varepsilon}(x,a)|
  \phi(x,a) \, dx \geq 0,
\end{multline*}
where~$g[\phi]\equiv
\left(\Lambda^\alpha-\varepsilon \partial_x^2\right) \phi$. Passing to
the limit as~$a \rightarrow 0$, thanks to the continuity with values
in~$L^1_{loc}(\R)$ of~$u^\varepsilon$ and~$\widetilde{u^\varepsilon}$
in Theorem~\ref{th2.1}, one can prove that for all non-negative $\phi \in
\mathcal{D}(\R \times [0,{+\infty}))$ 
\begin{multline}\label{ineq:difference}
  \int_{\R} \int_0^{+\infty} |u^\varepsilon-\widetilde{u^\varepsilon}| \left(\phi_t+L|\phi_x|-g [\phi] \right)\,dxdt\\
  +\int_{\R} |u_0(x)-\widetilde{u}_0(x)| \phi(x,0) \, dx \geq 0.
\end{multline}

This is almost the same equation as that in \cite[equation
(4.11)]{A07} with the diffusive operator
$g=\Lambda^\alpha-\varepsilon \partial_x^2$ instead
of~$g=\Lambda^\alpha$. Hence, we can argue exactly as
in~\cite[Subsection 4.2]{A07} replacing the kernel of~$\Lambda^\alpha$
by the kernel of the new
operator~$\Lambda^\alpha-\varepsilon \partial_x^2$. This gives the
desired inequality~\eqref{est:fi} in place of the
inequality~\cite[equation (3.1)]{A07}.
\end{proof}

\begin{proof}[Proof of Theorem~\ref{thm:convg}]
Now, we are in a position to prove the convergence result in Theorem
\ref{thm:convg}. The proof follows two steps: first we show the
relative compactness of the family of functions $\mathcal{F} \equiv
\{\ue :\varepsilon \in (0,1]\}$ and, next, we pass to the limit in
entropy inequalities.


{\it Step 1: compactness.} Let us prove that 
\begin{equation}\label{rel-comp}
\mbox{$\mathcal{F}$ is relatively compact in $F \equiv C([0,T];L^1([-R,R]))$}
\end{equation}
for all $T,R>0$.
The space $F$ being a Banach space, the statement \eqref{rel-comp} 
is equivalent to the precompactness of $\mathcal{F}$: 
\begin{equation}\label{prop:pre}
\begin{split}
\forall \mu>0\quad &\exists \mathcal{F}_\mu \subseteq F\quad \mbox{relatively compact such that }\\
&\lim_{\mu \rightarrow 0} \sup_{\ue \in \mathcal{F}} \mbox{dist}_F(\ue,\mathcal{F}_\mu)=0.
\end{split}
\end{equation} 
To construct $\mathcal{F}_\mu$, we consider
an approximation of the Dirac mass
$$
\rho_\mu(x)\equiv \mu^{-1} \rho(\mu^{-1} x) 
$$
with a smooth, non-negative function $\rho=\rho(x)$, supported in
$ [-1,1]$ and such that
$\int_\R \rho(x)\,dx=1$.
Then we define 
$$
\mathcal{F}_\mu\equiv \left\{\ue_\mu : \varepsilon\in (0,1] \right\},
$$ 
where $\ue_\mu\equiv \ue \ast_x \rho_\mu$ and $\ast_x$ denotes the 
convolution product with respect to  the space variable. 

First, we have to prove that $\mathcal{F}_\mu$ is relatively compact in $F$. 
By estimate \eqref{est:inf}, it is clear that
\begin{equation}\label{esti-mu0}
\| \ue_\mu\|_\infty \leq \| u_0\|_\infty \quad \mbox{and} \quad 
\| \partial_x \ue_\mu \|_\infty \leq \| u_0 \|_\infty \| \partial_x \rho_\mu \|_1 .
\end{equation}
Moreover, using  equation \eqref{uge1} satisfied by $\ue$ we obtain
\begin{equation}\label{eqn:mu}
\partial_t \ue_\mu=-\Lambda^\alpha \ue_\mu
+\varepsilon \partial_{x}^2 \ue_\mu-( f(\ue))_x \ast_x \rho_\mu=0.
\end{equation}
Applying the equalities
 $\Lambda^\alpha \ue_\mu=\Lambda^\alpha (\ue \ast_x \rho_\mu)
= \ue \ast_x (\Lambda^\alpha \rho_\mu)$ we see that
\begin{equation*}
  \|\Lambda^\alpha \ue_\mu\|_\infty  \leq \| \ue\|_\infty \|\Lambda^\alpha \rho_\mu \|_1 
  \leq \| u_0 \|_\infty \|\Lambda^\alpha \rho_\mu \|_1.
\end{equation*}
The same way, one can prove that 
$$
\|\partial_x \ue_\mu\|_\infty \leq \| u_0 \|_\infty \|\partial_{x}^2 \rho_\mu \|_1
\mbox{ and } \|(f(\ue))_x \ast_x \rho_\mu\|_\infty \leq C(\| u_0
\|_\infty) \|\partial_x \rho_\mu \|_1.
$$
Consequently, it follows from equation \eqref{eqn:mu} that for every
fixed $\mu>0$, the time derivative of $\ue_\mu$ is bounded
independently of $\varepsilon\in (0,1]$.  By \eqref{esti-mu0} and the
Ascoli-Arzel\`a Theorem, we infer that $\mathcal{F}_\mu$ is relatively
compact in $C_b([-R,R] \times [0,T])$ and, \textit{a fortiori}, in
$F$.

Next, we have to prove that 
$\lim_{\mu \rightarrow 0} \sup_{\ue \in \mathcal{F}} \mbox{dist}_F(\ue,\mathcal{F}_\mu)=0$. 
Applying Theorem \ref{thm:fi} to the following simple inequality
$$
\|\ue(t) -\ue_\mu(t)\|_{L^1([-R,R])} \leq \int_{-R}^R \int_{-\mu}^\mu |\ue(x,t)-\ue(x-y,t)| \rho_\mu(y) \,dxdy
$$
 we get
\begin{equation*}
\begin{split}
\|\ue(t) -\ue_\mu(t)\|_{L^1([-R,R])} & \leq  \sup_{|y| \leq \mu} \int_{-R}^R |\ue(x,t)-\ue(x-y,t)| \, dx, \\
 & \leq  \sup_{|y| \leq \mu} \int_{-R-Lt}^{R+Lt} 
S_\alpha^\varepsilon(t) v_0^y(x)\,dx,
\end{split}
\end{equation*}
where $v_0^y(x)=|u_0(x)-u_0(x-y)|$.  Consequently, by Lemma
\ref{lem:translation} in Appendix~\ref{appendix}, we see that there
exists a modulus of continuity~$\omega$ such that for all $r>0$ and
$\varepsilon \in (0,1]$
$$
\|\ue -\ue_\mu\|_{F} \leq \sup_{|y| 
\leq \mu} \int_{-R-LT-r}^{R+LT+r} v_0^y(x)\,dx + \|v_0^y\|_\infty \omega(1/r).
$$
The continuity of the translation in $L^1$ implies that
$$
\lim_{\mu \rightarrow 0} \sup_{|y| \leq \mu} \int_{-R-LT-r}^{R+LT+r}
v_0^y(x)\,dx=0.
$$
Hence, it is clear that $\lim_{\mu \rightarrow 0} \sup_{\varepsilon
  \in (0,1]} \|\ue -\ue_\mu\|_{F}=0$, which proves \eqref{prop:pre}
and thus \eqref{rel-comp}.

{\it Conclusion: passage to the limit.}  It follows from the first
step that there exists $v \in C([0,{+\infty});L^1_{loc}(\R))$ such that
$\lim_{\varepsilon \rightarrow 0} \ue = v$ (up to a subsequence) in
$C([0,T];L_{loc}^1(\R))$ for all $T>0$.  Passing to another
subsequence, if necessary, we can assume that $\ue \rightarrow v$
a.e. From inequality \eqref{est:inf}, we deduce that $v \in
L^\infty(\R \times (0,{+\infty}))$.  What we have to prove is that $v=u$,
however, by the uniqueness of entropy solutions ({\it cf.} Theorem
\ref{thm:entropy}), it suffices to show that $v$ is an entropy
solution to \eqref{eg1.1}--\eqref{eg1.2}.

Let $\eta \in C^2(\R)$ be convex, $\phi'=\eta' f'$ and
$r>0$. Integrating by parts the term~$-\varepsilon \int_{\R}
\int_a^{+\infty} \left(\eta (u^\varepsilon) \right)_{x} \varphi_x\, dxdt$
in~\eqref{eqn:entropy:reg} and passing to the limit~$a \rightarrow 0$
in this inequality, we get
\begin{equation*}
\begin{split}
  \int_{\R} \int_0^{+\infty}
  &\left(\eta(\ue) \varphi_t +\phi(\ue) \varphi_x-\eta(\ue) \Lambda_r^{(\alpha)} \varphi 
-\varphi \eta'(\ue)  \; \Lambda_r^{(0)} \ue \right)\,dxdt \\
  &+\int_{\R} \eta(u_0(x)) \varphi(x,0)\,dx \geq -\varepsilon
  \int_{\R} \int_0^{+\infty} \eta(\ue) \varphi_{xx} \,dxdt.
\end{split}
\end{equation*}
Finally, let us recall that $\ue \rightarrow v$ a.e. as $\varepsilon
\rightarrow 0$ and that $\ue$ is bounded in $L^\infty$-norm by
$\|u_0\|_\infty$.  Hence, the Lebesgue dominated convergence theorem
allows us to pass to the limit, as $\varepsilon \to 0$, in the
inequality above and to deduce that
\begin{equation*}
\begin{split}
  \int_{\R} \int_0^{+\infty} \Big(\eta(v) \varphi_t +\phi(v) \varphi_x
  -&\eta(v) \Lambda_r^{(\alpha)} \varphi
  -\varphi \eta'(v)  \; \Lambda_r^{(0)} v \Big)\,dxdt \\
  &+\int_{\R} \eta(u_0(x)) \varphi(x,0)\,dx \geq 0.
\end{split}
\end{equation*}
Hence, according to Definition \ref{def:entropy} and Theorem
\ref{thm:entropy}, the function $v$ is the unique entropy solution to
\eqref{eg1.1}--\eqref{eg1.2}.  The proof of Theorem \ref{thm:convg} is
complete.  
\end{proof}

\section{Additional technical lemmata}\label{appendix}

\begin{lemma}\label{lem:translation}
  There exists a modulus of continuity $\omega$ such that for all $v_0
  \in L^\infty(\R)$, all $T,R,r>0$, and all $ \varepsilon\in (0,1]$,
  we have
\begin{equation*}
  \sup_{t \in [0,T]} \int_{-R-Lt}^{R+Lt} S_\alpha^\varepsilon(t) |v_0|(x)\,dx 
  \leq \int_{-R-LT-r}^{R+LT+r} |v_0(x)|\,dx 
  + \|v_0\|_\infty \omega\left(1/r\right).
\end{equation*}
\end{lemma}

\Proof First, we write
\begin{equation}\label{est:comp-tech}
\begin{split}
  \sup_{t \in [0,T]} \int_{-R-Lt}^{R+Lt} S_\alpha^\varepsilon(t) &|v_0|(x)\,dx \\
  & =  \sup_{t \in [0,T]} \int_{-R-Lt}^{R+Lt} p_\alpha(t) \ast p_2(\varepsilon t) \ast |v_0|(x)\,dx \\
  & \leq \sup_{s \in [0,T]} \sup_{t \in [0,T]} \int_{-R-Lt}^{R+Lt}
  p_\alpha(t) \ast p_2(\varepsilon s) \ast |v_0|(x)\,dx.
\end{split}
\end{equation}
Now, for every  $s \in [0,T]$, we estimate  from above the following function
$$
M(s)\equiv \sup_{t \in [0,T]} \int_{-R-Lt}^{R+Lt} p_\alpha(t) \ast w_0(x)\,dx,
$$
where $w_0\equiv p_2(\varepsilon s) \ast |v_0|$.  Using properties of
the kernel $p_\alpha$ and its self-similarity (see \eqref{homogeneity}) we
obtain
\begin{equation*}
\begin{split}
  \int_{-R-Lt}^{R+Lt} p_\alpha(t) \ast w_0(x)\,dx
  =&   \int_{|x| \leq R+Lt} \int_{|y| \leq r/2} p_\alpha(y,t)w_0(x-y) \,dx dy\\
  &  +\int_{|x| \leq R+Lt} \int_{|y| \geq r/2} p_\alpha(y,t)w_0(x-y)\, dx dy\\
  \leq & \|p_\alpha(t)\|_1 \int_{-R-Lt-r/2}^{R+Lt+r/2} |w_0(x)|\,dx \\
  & + \|w_0\|_\infty 2(R+Lt) \int_{|y| \geq r/2} p_\alpha(y,t)\, dy\\
  = & \int_{-R-Lt-r/2}^{R+Lt+r/2} |w_0(x)|\,dx\\
  &+ \|w_0\|_\infty 2(R+Lt) \int_{|x| \geq t^{-\frac{1}{\alpha}} r/2}
  p_\alpha(x,1)\,dx.
\end{split}
\end{equation*}
Computing the supremum with respect to $t \in [0,T]$ we infer that
$$
M(s) \leq \int_{-R-LT-r/2}^{R+LT+r/2} |w_0(x)|\,dx+ 
\|w_0\|_\infty \omega_\alpha(1/r),
$$
where $\omega_\alpha:[0,+\infty)\to (0,+\infty)$ is defined by
$$
\omega_\alpha(1/r)\equiv (2R+2LT) \int_{|x| \geq T^{-\frac{1}{\alpha}}
  r/2} p_\alpha(x,1)\,dx.
$$ 
It is clear that the modulus of continuity $\omega_\alpha$ is
non-decreasing and satisfies 
$$ 
\lim_{r \rightarrow +\infty} \omega_\alpha(1/r)=0.
$$ 
Finally, since
$\|w_0\|_\infty=\|p_2(\varepsilon s) \ast |v_0| \|_\infty \leq
\|v_0\|_\infty$, we obtain
$$
M(s) \leq \int_{-R-LT-r/2}^{R+LT+r/2} |w_0(x)|\,dx
+ \|v_0\|_\infty \omega_\alpha(1/r).
$$

Analogous  computations show now that
\begin{equation*}
\begin{split}
  \int_{-R-LT-r/2}^{R+LT+r/2} |w_0(x)|\,dx
  & =  \int_{-R-LT-r/2}^{R+LT+r/2} p_2(\varepsilon s) \ast |v_0|(x)\,dx\\
  & \leq \int_{-R-LT-r}^{R+LT+r} |v_0(x)|\,dx
  +\|v_0\|_\infty \omega_2(\sqrt{\varepsilon}/r)\\
  & \leq \int_{-R-LT-r}^{R+LT+r} |v_0(x)|\,dx+\|v_0\|_\infty
  \omega_2(1/r),
\end{split}
\end{equation*}
because $\varepsilon \leq 1$. 

Finally, with the new modulus of continuity
$\omega\left(1/r\right)\equiv \omega_\alpha(1/r)+\omega_2(1/r), $ we
have
$$
M(s) \leq \int_{-R-LT-r}^{R+LT+r} |v_0(x)|\,dx+\|v_0\|_\infty \omega(1/r).
$$
Coming back to inequality  \eqref{est:comp-tech}, we complete the proof of
 Lemma \ref{lem:translation}. 
\endProof
 
\begin{lemma}\label{lem:taylor}
  Let~$I$ be an open interval of~$\R$ and~$u \in W^{1,\infty}(I)$ be
  such that~$u_x \in BV(I)$. Then, for a.e.~$x \in I$ and all~$z \in
  I-x$, we have
$$
u(x+z)=u(x)+u_x(x) z+\int_{I_{x,z}} \; |x+z-y| \; u_{yy}(dy),
$$
where~$I_{x,z}\equiv(x,x+z)$ if~$z>0$ and~$I_{x,z}\equiv (x+z,x)$ if
not.
\end{lemma}
\begin{proof}
  We can reduce to the case $I=(a,b)$ with $a,b \in \R$. 
  Let us assume without loss of generality that~$z>0$. Since~$u_x \in
  BV(I)$, the function~$\widetilde{u}_x(x)\equiv c+\int_{(a,x]}
  u_{yy}(dy)$ is an a.e. representative of~$u_x$, where~$c$ is the trace
  of~$u_x$ on the left boundary of~$I$. The trace of~$u_x \in
  BV(I_{x,z})$ onto~$\{x\}$ is equal to~$\widetilde{u}_x(x)$,
  because~$\{x\}$ is the left boundary of~$I_{x+z}$. Simple
  integration by parts formulas now give
\begin{eqnarray*}
u(x+z) & = & u(x)+\int_{I_{x,z}} u_y(y) dy\\
 & = & u(x)-\int_{I_{x,z}} (y-x-z) u_{yy}(dy)+\widetilde{u}_x(x)z.
\end{eqnarray*}
The proof is complete.
\end{proof}

\end{document}